\pdfoutput=1
\documentclass[12pt]{article}%
\usepackage{graphicx}
\usepackage{color}
\usepackage{amsmath}
\usepackage{amsfonts}
\usepackage{amssymb}
\usepackage{amssymb}%
\usepackage{hyperref}
\providecommand{\U}[1]{\protect\rule{.1in}{.1in}}
\providecommand{\U}[1]{\protect\rule{.1in}{.1in}}
\providecommand{\U}[1]{\protect\rule{.1in}{.1in}}
\newtheorem{theorem}{Theorem}

\newtheorem{algorithm}[theorem]{Algorithm}

\newtheorem{condition}[theorem]{Condition}

\numberwithin{equation}{section}
\numberwithin{theorem}{section}
\newtheorem{definition}[theorem]{Definition}
\newtheorem{example}[theorem]{Example}

\newtheorem{lemma}[theorem]{Lemma}
\newtheorem{notation}[theorem]{Notation}

\newtheorem{proposition}[theorem]{Proposition}
\newtheorem{remark}[theorem]{Remark}

\newenvironment{proof}[1][Proof]{\textbf{#1.} }{\ \rule{0.5em}{0.5em}}
\begin{document}

\title{\textbf{An algorithm for solving the variational inequality problem over the
fixed point set of a quasi-nonexpansive operator in Euclidean space}}
\author{Andrzej Cegielski$^{1}$, Aviv Gibali$^{2}$, Simeon Reich$^{2}$ and
Rafa\l \ Zalas$^{1}$\bigskip\\$^{1}$Faculty of Mathematics, Computer Science and Econometrics,\\University of Zielona G\'{o}ra, Zielona G\'{o}ra, Poland\bigskip\\$^{2}$Department of Mathematics,\\The Technion - Israel Institute of Technology\\Technion City, 32000 Haifa, Israel}
\date{July 25, 2012, Revised December 17, 2012.}
\maketitle

\begin{abstract}
This paper is concerned with the variational inequality problem\newline%
\textit{VIP}$(\mathcal{F},\operatorname{Fix}\left(  T\right)  )$: find
$\bar{u}\in\operatorname{Fix}\left(  T\right)  $ such that $\langle
\mathcal{F}(\bar{u}),z-\bar{u}\rangle\geq0$ for all $z\in\operatorname{Fix}%
\left(  T\right)  $, where $T:%
\mathbb{R}
^{n}\rightarrow%
\mathbb{R}
^{n}$ is quasi-nonexpansive, $\operatorname{Fix}(T)$ is its nonempty fixed
point set, and $\mathcal{F}:%
\mathbb{R}
^{n}\rightarrow%
\mathbb{R}
^{n}$ is monotone.
We propose, in particular, an algorithm which entails, at each step,
projecting onto a suitably chosen half-space, and prove that the sequences it
generates converge to the unique solution of the VIP. We also present an
application of our result to a hierarchical optimization problem.\smallskip

\smallskip


\end{abstract}

\section{\textbf{Introduction}}

The classical variational inequality problem (VIP) is to find a point
$x^{\ast}\in S$ such that
\begin{equation}
\left\langle \mathcal{F}(x^{\ast}),x-x^{\ast}\right\rangle \geq0\text{ for all
}x\in S\text{,} \label{eq:Classicvip}%
\end{equation}
where $S\subseteq%
\mathbb{R}
^{n}$ is nonempty, closed and convex, $\mathcal{F}:%
\mathbb{R}
^{n}\rightarrow%
\mathbb{R}
^{n}$ is a given operator,$\ $and $\langle\cdot,\cdot\rangle$ denotes the
inner product in $%
\mathbb{R}
^{n}$. This problem, denoted by VIP$(\mathcal{F},S)$, is a fundamental problem
in Optimization Theory because many optimization problems can be translated
into VIPs. The VIP was intensively studied in the last decades; see,
\textit{e.g.}, the two-volume book by Facchinei and Pang \cite{Facc}, and the
review papers by Noor \cite{noor04} and by Xiu and Zhang \cite{naihua03}. Some
algorithms for solving (\ref{eq:Classicvip}) fit into the framework of the
following general iterative scheme:%
\begin{equation}
x^{k+1}=P_{S}(x^{k}-\tau_{k}\mathcal{F}(x^{k}))\text{,} \label{eq:aus}%
\end{equation}
where $\tau_{k}\geq0$ and $P_{S}$ is the metric projection\ operator onto $S$
(see Auslender \cite{Auslender} and consult \cite[Volume 2, Subsection
12.1]{Facc} for more details).

Such methods are particularly useful when the set $S$ is simple enough to
project onto. However, in general, one has to solve a minimization problem
(evaluation of the metric projection onto $S$) at each iterative step in order
to get the next iterate. In this case the efficiency of method (\ref{eq:aus})
may be seriously affected. Fukushima \cite{F86} developed a method that
overcomes this obstacle by replacing the metric projection onto the set
$S$\ by a subgradient projection which is easier to calculate. Censor and
Gibali present in \cite{CG} the $\delta$-algorithmic scheme which generalizes
the Auslender and Fukushima algorithms in the sense that there is some
\textquotedblleft freedom\textquotedblright\ in choosing the hyperplane onto
which one projects.

In this paper we present a natural extension of this $\delta$-algorithmic
scheme (Algorithm \ref{alg:ac} below). Given an operator $T:%
\mathbb{R}
^{n}\rightarrow%
\mathbb{R}
^{n}$, we denote by $\operatorname*{Fix}(T):=\left\{  x\in%
\mathbb{R}
^{n}\mid T(x)=x\right\}  $ the fixed point set of $T$. It is well known that
$\operatorname*{Fix}(T)$ is closed and convex if $T$ is quasi-nonexpansive
(see, \textit{e.g.}, \cite[Proposition 2.6 (ii)]{BC01}). Observe that the
feasible set $S$ of the VIP\ in (\ref{eq:Classicvip}) can always be
represented as the fixed point set of some operator, say,
$S=\operatorname*{Fix}\left(  P_{S}\right)  $. Following this idea, Yamada and
Ogura \cite{Y2004a} considered the variational inequality problem
VIP$(\mathcal{F},\operatorname*{Fix}(T))$, which calls for finding a point
$x^{\ast}\in\operatorname*{Fix}(T)$ such that%
\begin{equation}
\left\langle \mathcal{F}(x^{\ast}),x-x^{\ast}\right\rangle \geq0\text{ for all
}x\in\operatorname*{Fix}(T)\text{.} \label{eq:vip}%
\end{equation}
In the case where $T$ is quasi-nonexpansive and so-called quasi-shrinking, an
algorithm for solving (\ref{eq:vip}) in a real Hilbert space was proposed in
\cite{Y2004a} under the conditions of Lipschitz continuity and strong
monotonicity of $\mathcal{F}$. The iterative step of the method is as follows:%
\begin{equation}
x^{k+1}=T(x^{k})-\lambda_{k+1}\mathcal{F}(T(x^{k}))\text{,} \label{eq:yamada}%
\end{equation}
where $\left\{  \lambda_{k}\right\}  _{k=0}^{\infty}$ is a nonnegative
sequence which satisfies certain conditions. As a matter of fact, Yamada and
Ogura \cite[Theorem 5]{Y2004a} showed that (\ref{eq:yamada}) could be applied
to more general cases with weaker monotonicity assumptions, such as
paramonotonicity (see e.g., \cite{yyy}).

In this paper we present a method for solving the VIP$(\mathcal{F}%
,\operatorname*{Fix}(T))$ when the operator $T$ is merely quasi-nonexpansive.
This method generalizes the earlier results of Auslender and Fukushima, as
well as the $\delta$-algorithmic scheme. In addition, we present the
relationship between our algorithm and the Yamada--Ogura method. Note that
several authors have considered the VIP$(\mathcal{F},\operatorname*{Fix}(T))$
for a quasi-nonexpansive operator $T$ and proposed methods similar to
(\ref{eq:yamada}), where $T$ is replaced by a sequence of quasi-nonexpansive
operators $T_{k}$ with the property $\operatorname*{Fix}\left(  T\right)
\subseteq\bigcap_{k\geq0}\operatorname*{Fix}\left(  T_{k}\right)  $ (see
\cite{Bau96, CZ11, Hir06, Y2004a}).

Our paper is organized as follows. Section \ref{Sec:pre} is devoted to some
definitions and preliminary results. Our algorithm is described in Section
\ref{Sec:alg} and analyzed in Section \ref{Sec:con}. Finally, in Section
\ref{Sec:app} we present an application of our results to a hierarchical
optimization problem.

\section{Preliminaries\label{Sec:pre}}

In this section we recall some definitions and properties of several classes
of operators.

\begin{definition}%
\rm\
Let $T:%
\mathbb{R}
^{n}\rightarrow%
\mathbb{R}
^{n}$ be an operator with a fixed point. The operator $T$ is called:

\begin{itemize}
\item[(i)] $\alpha$\textit{-}\texttt{strongly quasi-nonexpansive} ($\alpha
$-SQNE), where $\alpha\geq0$, if for all $(x,w)\in%
\mathbb{R}
^{n}\times\operatorname*{Fix}(T)$, we have%
\begin{equation}
\left\Vert T\left(  x\right)  -w\right\Vert ^{2}\leq\left\Vert x-w\right\Vert
^{2}-\alpha\left\Vert x-T\left(  x\right)  \right\Vert ^{2}\text{;}%
\end{equation}
If $\alpha>0,$ then we say that $T$ is \texttt{strongly quasi-nonexpansive} (SQNE);

\item[(ii)] \texttt{Firmly quasi-nonexpansive} (FQNE) if it is $1$-SQNE;

\item[(iii)] \texttt{Quasi-nonexpansive}\textit{ }if it is $0$-SQNE, i.e.,%
\begin{equation}
\Vert T\left(  x\right)  -w\Vert\leq\Vert x-w\Vert\text{ for all }(x,w)\in%
\mathbb{R}
^{n}\times\operatorname*{Fix}(T)\text{;}%
\end{equation}

\item[(iv)] \texttt{Nonexpansive} if%
\begin{equation}
\Vert T\left(  x\right)  -T\left(  y\right)  \Vert\leq\Vert x-y\Vert\text{ for
all }x,y\in%
\mathbb{R}
^{n}\text{.}%
\end{equation}

\end{itemize}
\end{definition}

The class of quasi-nonexpansive operators was denoted by Crombez \cite[p.
161]{Crombez05} by $\mathcal{F}^{0}$. An important subset of $\mathcal{F}^{0}%
$, namely the $\mathfrak{T}$-class operators, was introduced and investigated
by Bauschke and Combettes\textit{ }\cite{BC01}, and by\textit{ }%
Combettes\textit{ }\cite{Co}. The operators in this class were named
\textit{directed operators }in Zaknoon \cite{Z} and further used under this
name in \cite{cs08}. Cegielski \cite{Ceg08} studied these operators under the
name \textit{separating operators}. Since both \textit{directed }%
and\textit{\ separating }are key words of other, widely-used, mathematical
entities, Cegielski and Censor have recently introduced the term
\textit{cutter operators} \cite{cc11}, or cutters in short. This class
coincides with the class $\mathcal{F}^{1}$ (see \cite{Crombez05}), with the
class of $1$-SQNE operators (see \cite[Theorem 2.1.39]{Ceg12}) and with the
class DC$_{\boldsymbol{p}}$ for $\boldsymbol{p}=-1$ \cite{mp08}. The term
\textit{firmly quasi-nonexpansive} (FQNE) for $\mathfrak{T}$-class operators
was used by Yamada and Ogura \cite[Section B]{Yamada, Y2004a} and by
M\u{a}ru\c{s}ter \cite{Maruster05}. This class of operators is fundamental
because it contains several types of operators commonly found in various areas
of applied mathematics, such as the metric projections, subgradient
projections and the resolvents of maximal monotone operators (see
\cite{BC01}). The formal definition of the $\mathfrak{T}$-class in Euclidean
space is as follows.

\begin{definition}%
\rm\
An operator $T:%
\mathbb{R}
^{n}\rightarrow%
\mathbb{R}
^{n}$ is called a \texttt{cutter} ($T\in\mathfrak{T}$) if $\operatorname*{Fix}%
(T)\neq\emptyset$ and%
\begin{equation}
\left\langle T\left(  x\right)  -x,T\left(  x\right)  -w\right\rangle
\leq0\text{\ for all }(x,w)\in%
\mathbb{R}
^{n}\times\operatorname*{Fix}(T)\text{.} \label{def directed}%
\end{equation}

\end{definition}

For $x,y\in%
\mathbb{R}
^{n}$, we denote
\begin{equation}
H(x,y):=\left\{  u\in%
\mathbb{R}
^{n}\mid\left\langle u-y,x-y\right\rangle \leq0\right\}  \text{.}
\label{eq:halfspaceH}%
\end{equation}
If $x\neq y$, then $H(x,y)$ is a half-space. It is easy to see that $T$ is a
cutter if and only if%
\begin{equation}
\operatorname*{Fix}(T)\subseteq H(x,T(x))\text{ for all }x\in%
\mathbb{R}
^{n}\text{.}%
\end{equation}
This property is illustrated in Figure \ref{fig:1}.%
\begin{figure}
[h]
\begin{center}
\includegraphics[
natheight=7.499600in,
natwidth=9.999800in,
height=2.2511in,
width=2.994in
]%
{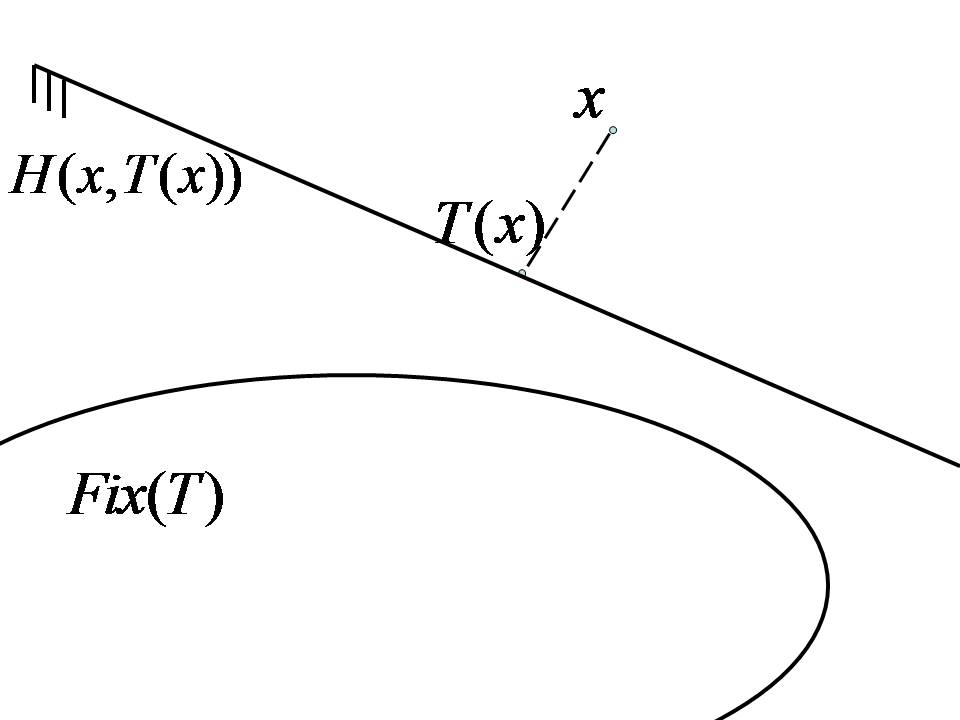}%
\caption{Property of cutters}%
\label{fig:1}%
\end{center}
\end{figure}

\begin{definition}%
\rm\
Let $T:%
\mathbb{R}
^{n}\rightarrow%
\mathbb{R}
^{n}$ be a cutter and $\alpha\in\lbrack0,2]$. The operator $T_{\alpha
}:=I+\alpha(T-I)$, where $I$ denotes the identity operator, is called an
$\alpha$\textit{-}\texttt{relaxed cutter.}
\end{definition}

Bauschke and Combettes \cite{BC01} established the following two properties of cutters.

\begin{itemize}
\item[(i)] The set $\operatorname*{Fix}(T)$ of all fixed points of a cutter
operator $T$ is closed and convex because%
\begin{equation}
\operatorname*{Fix}(T)=\cap_{x\in%
\mathbb{R}
^{n}}H\left(  x,T\left(  x\right)  \right)  \text{.}%
\end{equation}

\item[(ii)] If $T\in\mathfrak{T}$ and $\alpha\in\lbrack0,1]$, then $T_{\alpha
}\in\mathfrak{T}$.
\end{itemize}

One can easily verify the following characterization of $\alpha$-relaxed
cutter operators $U$ ($U=T_{\alpha}$):%
\begin{equation}
\alpha\left\langle U\left(  x\right)  -x,w-x\right\rangle \geq\Vert U\left(
x\right)  -x\Vert^{2}\text{\ for all }(x,w)\in%
\mathbb{R}
^{n}\times\operatorname*{Fix}(T)\text{.}%
\end{equation}

\begin{theorem}
\label{Lem:2}Let $\alpha\in(0,2]$. An operator $U:%
\mathbb{R}
^{n}\rightarrow%
\mathbb{R}
^{n}$ is a cutter if and only if $U_{\alpha}$ is $\left(  \left(
2-\alpha\right)  /\alpha\right)  $-strongly quasi-nonexpansive
\end{theorem}

\begin{proof}
See, \textit{e.g.}, \cite[Proposition 2.3 (ii)]{Co} and \cite[Theorem
2.1.39]{Ceg12}.
\end{proof}

\begin{notation}%
\rm\
Let $C$ be a nonempty, closed and convex subset of $%
\mathbb{R}
^{n}$, and let $T\colon\mathbb{R}^{n} \rightarrow\mathbb{R}^{n}$ be quasi-nonexpansive.

\begin{itemize}
\item[(i)] The \texttt{distance function} of a point $x\in%
\mathbb{R}
^{n}$ to $C$ is defined by%
\begin{equation}
\operatorname*{dist}(x,C):=\inf\{\Vert x-z\Vert\mid z\in C\}\text{.}%
\end{equation}

\item[(ii)] For$\ r\geq0$ we define the subsets%
\begin{equation}
C^{r}:=\left\{  u\in%
\mathbb{R}
^{n}\mid\operatorname*{dist}(u,C)\geq r\right\}
\end{equation}
and%
\begin{equation}
C_{r}:=\{u\in%
\mathbb{R}
^{n}\mid\operatorname*{dist}(u,C)\leq r\}\text{.}%
\end{equation}

\item[(iii)] Define the function $D\colon\lbrack0,\infty)\rightarrow
\lbrack0,\infty]$ in the following way:
\begin{equation}
D(r):=%
\begin{cases}
\underset{u\in(\operatorname{Fix}(T))^{r}\cap C}{\inf}\left(
\operatorname{dist}(u,\operatorname{Fix}(T))-\operatorname{dist}%
(T(u),\operatorname{Fix}(T))\right) \\
\text{if $(\operatorname{Fix}(T))^{r}\cap C\neq\emptyset$,}\\
+\infty\text{, otherwise.}%
\end{cases}
\label{eq:2.12}%
\end{equation}

\item[(iv)] Denote the metric projection onto $\operatorname{Fix}(T)$ by $R$,
i.e., $R:=P_{\operatorname{Fix}(T)}$.
\end{itemize}
\end{notation}

It is well known that, for a convex subset $C$, the (continuous) distance
function $\operatorname*{dist}(\cdot,C)$ is convex. Consequently, $C_{r}$ is
closed and convex as a sublevel set of a convex function. The continuity of
$\operatorname{dist}(\cdot,C)$ implies that the subset $C^{r}$ is also closed.

In what follows we assume that $T:\mathbb{R}^{n}\rightarrow\mathbb{R}^{n}$ is
quasi-nonexpansive and $C\subseteq\mathbb{R}^{n}$ is a closed and convex set
such that $\operatorname{Fix}(T)\cap C\neq\emptyset$.

Parts (i) and (ii) of the following theorem can be found in \cite[Lemma
2]{Y2004a}.

\begin{proposition}
\label{proposition:D}The function $D:[0,\infty)\rightarrow\lbrack0,\infty]$
defined by $\mathrm{(\ref{eq:2.12})}$ has the following properties:

\begin{itemize}
\item[$\mathrm{(i)}$] $D(0)=0$ and $D(r)\geq0$ for all $r\geq0$;

\item[$\mathrm{(ii)}$] if $r_{1}\geq r_{2}\geq0,$ then $D(r_{1})\geq D(r_{2})$;

\item[$\mathrm{(iii)}$] $D(\operatorname*{dist}(x,\operatorname*{Fix}%
(T)))\leq\Vert x-T\left(  x\right)  \Vert$ for all $x\in C$.
\end{itemize}
\end{proposition}

\begin{proof}
(i) Let $r\geq0$. We prove that $D(r)\geq0$. The inequality is clear if
$(\operatorname{Fix}(T))^{r}\cap C=\emptyset$. Now suppose that
$(\operatorname{Fix}(T))^{r}\cap C\neq\emptyset$. Then the definition of the
metric projection and the quasi-nonexpansivity of $T$ yield, for any
$x\in(\operatorname{Fix}(T))^{r}\cap C$,
\begin{equation}
\left\Vert x-R\left(  x\right)  \right\Vert -\left\Vert T\left(  x\right)
-R\left(  T\left(  x\right)  \right)  \right\Vert \geq\left\Vert x-R\left(
x\right)  \right\Vert -\left\Vert T\left(  x\right)  -R\left(  x\right)
\right\Vert \geq0.
\end{equation}
Consequently, $D(r)\geq0$. Let $x\in\operatorname{Fix}(T)\cap C$. By the
quasi-nonexpansivity of $T$, we have $T\left(  x\right)  =x$ and
\begin{equation}
D(0)\leq\left\Vert x-R\left(  x\right)  \right\Vert -\left\Vert T\left(
x\right)  -R\left(  T\left(  x\right)  \right)  \right\Vert =0,
\end{equation}
which together with the first part proves that $D(0)=0$.

(ii) Let $r_{1}\geq r_{1}\geq0$. Then, of course, $(\operatorname{Fix}%
(T))^{r_{2}}\subseteq(\operatorname{Fix}(T))^{r_{1}}$ and so the property is clear.

(iii) Let $x\in C$ and $r=\operatorname*{dist}(x,\operatorname{Fix}(T))$. If
$r=0$, then $T\left(  x\right)  =x$ and, by (i), the assertion is obvious. Let
$r>0$. Then, of course, $x\in(\operatorname{Fix}(T))^{r}\cap C$ and, by the
definition of the metric projection and the triangle inequality, we have
\begin{align}
D(r)  &  \leq\left\Vert x-R\left(  x\right)  \right\Vert -\left\Vert T\left(
x\right)  -R\left(  T\left(  x\right)  \right)  \right\Vert \nonumber\\
&  \leq\left\Vert x-R\left(  T\left(  x\right)  \right)  \right\Vert
-\left\Vert T\left(  x\right)  -R\left(  T\left(  x\right)  \right)
\right\Vert \leq\left\Vert x-T\left(  x\right)  \right\Vert ,
\end{align}
and the proof is complete.
\end{proof}

Now we give two equivalent definitions of a quasi-shrinking operator.

\begin{definition}%
\rm\
\label{def:qs}(cf. \cite{Y2004a}) Let $C\subseteq%
\mathbb{R}
^{n}$ be closed and convex, and let $T:\mathbb{R}^{n}\rightarrow\mathbb{R}%
^{n}$ be a quasi-nonexpansive operator. The operator $T$ is called
\texttt{quasi-shrinking} on $C$ if $D(r)=0\Leftrightarrow r=0$.$\smallskip$
\end{definition}

\begin{definition}%
\rm\
\label{def:qs-1}Let $C\subseteq%
\mathbb{R}
^{n}$ be closed and convex, and let $T:%
\mathbb{R}
^{n}\rightarrow%
\mathbb{R}
^{n}$ be a quasi-nonexpansive operator. We say that $T$ is
\texttt{quasi-shrinking} on $C$ if for any sequence $\{u^{k}\}_{k=0}^{\infty
}\subseteq C$, the following implication holds:%
\begin{equation}
\underset{k\rightarrow\infty}{\lim}\left(  \Vert u^{k}-R\left(  u^{k}\right)
\Vert-\Vert T\left(  u^{k}\right)  -R\left(  T\left(  u^{k}\right)  \right)
\Vert\right)  =0\Rightarrow\underset{k\rightarrow\infty}{\lim}\Vert
u^{k}-R\left(  u^{k}\right)  \Vert=0\text{.}%
\end{equation}

\end{definition}

\begin{proposition}
Let $C\subseteq%
\mathbb{R}
^{n}$ be closed and convex, and let a quasi-nonexpansive operator $T:%
\mathbb{R}
^{n}\rightarrow%
\mathbb{R}
^{n}$ be such that $\operatorname*{Fix}(T)\cap C\neq\emptyset$. Then
Definitions $\mathrm{\ref{def:qs}}$ and $\mathrm{\ref{def:qs-1}}$ are equivalent.
\end{proposition}

\begin{proof}
Let $T$ be quasi-shrinking in the sense of Definition \ref{def:qs} and choose
$\{u^{k}\}_{k=0}^{\infty}\subseteq C$. Suppose that $\lim_{k\rightarrow\infty
}\Vert u^{k}-R\left(  u^{k}\right)  \Vert\neq0$. Then there exist a constant
$\varepsilon>0$ and a subsequence $\{u^{k_{j}}\}_{j=0}^{\infty}\subseteq
\{u^{k}\}_{k=0}^{\infty}$ such that $\Vert u^{k_{j}}-R\left(  u^{k_{j}%
}\right)  \Vert>\varepsilon$. Therefore $u^{k_{j}}\in(\operatorname*{Fix}%
(T))^{\varepsilon}\cap C$ and we have%
\begin{align}
&  \inf_{j\geq0}(\Vert u^{k_{j}}-R\left(  u^{k_{j}}\right)  \Vert-\Vert
T\left(  u^{k_{j}}\right)  -R\left(  T\left(  u^{k_{j}}\right)  \right)
\Vert)\nonumber\\
&  \geq\inf_{u\in(\operatorname*{Fix}(T))^{\varepsilon}\cap C}(\Vert
u-R\left(  u\right)  \Vert-\Vert T\left(  u\right)  -R\left(  T\left(
u\right)  \right)  \Vert)\nonumber\\
&  =D(\varepsilon)>0\text{.}%
\end{align}
Consequently, $\lim_{k\rightarrow\infty}\left(  \Vert u^{k}-R\left(
u^{k}\right)  \Vert-\Vert T\left(  u^{k}\right)  -R\left(  T\left(
u^{k}\right)  \right)  \Vert\right)  \neq0$ if it exists.

In the other direction, let $T$ be quasi-shrinking in the sense of Definition
\ref{def:qs-1}. Suppose that $D(r)=0$ for some $r\geq0$. Then there is a
sequence $\{u^{k}\}_{k=0}^{\infty}\subseteq(\operatorname*{Fix}(T))^{r}\cap C$
such that
\begin{equation}
\lim_{k\rightarrow\infty}(\Vert u^{k}-R\left(  u^{k}\right)  \Vert-\Vert
T\left(  u^{k}\right)  -R\left(  T\left(  u^{k}\right)  \right)
\Vert)=0\text{.}%
\end{equation}
By Definition \ref{def:qs-1}, we have
\begin{equation}
r\leq\lim_{k\rightarrow\infty}\operatorname*{dist}(u^{k},\operatorname*{Fix}%
(T))=\lim_{k\rightarrow\infty}\Vert u^{k}-R\left(  u^{k}\right)
\Vert=0\text{,}%
\end{equation}
i.e., $r=0$ and the proof is complete.
\end{proof}

\begin{definition}%
\rm\
Let $C\subseteq\mathbb{R}^{n}$ be closed and convex. An operator
$T:C\rightarrow%
\mathbb{R}
^{n}$ is called \texttt{closed} at $y\in%
\mathbb{R}
^{n}$ if for any sequence $\left\{  x^{k}\right\}  _{k=0}^{\infty}\subseteq
C$, we have%
\begin{equation}
\left(  \underset{k\rightarrow\infty}{\lim}x^{k}=x\in C\text{ and }%
\underset{k\rightarrow\infty}{\lim}T(x^{k})=y\right)  \Rightarrow
T(x)=y\text{.}%
\end{equation}
We say that the \textit{closedness principle} holds for an operator
$T:C\rightarrow%
\mathbb{R}
^{n}$ if $I-T$ is closed at $0$ (see \cite{Browder}), i.e.,
\begin{equation}
\left(  \underset{k\rightarrow\infty}{\lim}x^{k}=x\in C\text{ and }%
\underset{k\rightarrow\infty}{\lim}\left\Vert T(x^{k})-x^{k}\right\Vert
=0\right)  \Rightarrow x\in\operatorname*{Fix}\left(  T\right)  .
\end{equation}

\end{definition}

It is clear that in $%
\mathbb{R}
^{n}$ a continuous operator, in particular a nonexpansive one, satisfies the
closedness principle. Later in this section we give other examples of
operators satisfying the closedness principle (see Examples \ref{Ex:cut-cl},
\ref{Exmpl. subgradients} and \ref{def. E-delta op}).

\begin{proposition}
\label{prop:1}Let $C\subseteq%
\mathbb{R}
^{n}$ be closed, bounded and convex and let $T:%
\mathbb{R}
^{n}\rightarrow%
\mathbb{R}
^{n}$ be an operator with $\operatorname*{Fix}(T)\cap C\neq\emptyset$. If $T$
is SQNE (equivalently, an $\alpha$-relaxed cutter for some $\alpha\in(0,2)$)
and $I-T$ is closed at $0$, then $T$ is quasi-shrinking on $C$.
\end{proposition}

\begin{proof}
Let $r\geq0$ and $D(r)=0\,$. Then there is a sequence $\{u^{k}\}_{k=0}%
^{\infty}\subseteq(\operatorname*{Fix}(T))^{r}\cap C$ such that%
\begin{equation}
\lim_{k\rightarrow\infty}(\Vert u^{k}-R(u^{k})\Vert-\Vert T(u^{k})-R\left(
T(u^{k})\right)  \Vert)=0\text{.} \label{1}%
\end{equation}
By the quasi-nonexpansivity of $T$, the definition of the metric projection
and by (\ref{1}), we have
\begin{align}
0  &  \leq\Vert u^{k}-R(u^{k})\Vert-\Vert T(u^{k})-R(u^{k})\Vert\nonumber\\
&  \leq\Vert u^{k}-R(u^{k})\Vert-\Vert T(u^{k})-R\left(  T(u^{k})\right)
\Vert\rightarrow0\text{.}%
\end{align}
Consequently,
\begin{equation}
0\leq\Vert u^{k}-R(u^{k})\Vert-\Vert T(u^{k})-R(u^{k})\Vert\rightarrow
0\text{.} \label{2}%
\end{equation}
Since $T$ is SQNE, there is $\alpha>0$ such that
\begin{equation}
\Vert T(u^{k})-R(u^{k})\Vert^{2}\leq\Vert u^{k}-R(u^{k})\Vert^{2}-\alpha\Vert
T(u^{k})-u^{k}\Vert^{2}\text{.} \label{3}%
\end{equation}
Let $z\in\operatorname*{Fix}(T)$. By the boundedness of $\{u^{k}%
\}_{k=0}^{\infty}$, there exists $d>0$ such that $\Vert u^{k}-z\Vert\leq d$
for all $k\geq0$. Using the definition of the metric projection and the
quasi-nonexpansivity of $T$, we obtain%
\begin{align}
\Vert u^{k}-R(u^{k})\Vert+\Vert T(u^{k})-R(u^{k})\Vert &  \leq\Vert
u^{k}-z\Vert+\Vert u^{k}-R(u^{k})\Vert\nonumber\\
&  \leq2\Vert u^{k}-z\Vert\leq2d\text{.} \label{4}%
\end{align}
By (\ref{2}), (\ref{3}) and (\ref{4}), we now have
\begin{align}
&  \Vert T(u^{k})-u^{k}\Vert^{2}\nonumber\\
&  \leq\frac{1}{\alpha}(\Vert u^{k}-R(u^{k})\Vert^{2}-\Vert T(u^{k}%
)-R(u^{k})\Vert^{2})\nonumber\\
&  =\frac{1}{\alpha}(\Vert u^{k}-R(u^{k})\Vert-\Vert T(u^{k})-R(u^{k}%
)\Vert)(\Vert u^{k}-R(u^{k})\Vert+\Vert T(u^{k})-R(u^{k})\Vert)\nonumber\\
&  \leq\frac{2d}{\alpha}(\Vert u^{k}-R(u^{k})\Vert-\Vert T(u^{k}%
)-R(u^{k})\Vert)\rightarrow0\text{.}%
\end{align}
Consequently,
\begin{equation}
\lim_{k\rightarrow\infty}\Vert T(u^{k})-u^{k}\Vert=0\text{.}%
\end{equation}
Since $\{u^{k}\}_{k=0}^{\infty}$ is bounded, there exists a subsequence
$\{u^{k_{j}}\}_{j=0}^{\infty}$ of $\{u^{k}\}_{k=0}^{\infty}$ such that
\begin{equation}
\lim_{j\rightarrow\infty}u^{k_{j}}=u^{\ast}\text{.}%
\end{equation}
The closedness of $I-T$ at $0$ yields that $u^{\ast}\in\operatorname*{Fix}(T)$
and
\begin{equation}
r\leq\inf_{u\in(\operatorname*{Fix}(T))^{r}\cap C}\operatorname*{dist}%
(u,\operatorname*{Fix}\left(  T\right)  )\leq\lim_{j\rightarrow\infty}\Vert
u^{k_{j}}-u^{\ast}\Vert=0\text{,}%
\end{equation}
i.e., $r=0$, which proves that $T$ is quasi-shrinking, as asserted.
\end{proof}

\begin{remark}%
\rm\
The converse to Proposition \ref{prop:1} is not true. To see this, take
$C=\{u\in%
\mathbb{R}
^{n}\mid\langle a,u\rangle\leq\beta\}$ for some $a\neq0$, $\beta\in%
\mathbb{R}
$ and $T=2P_{C}-I$. Then $T$ is quasi-shrinking, but $T$ is not SQNE.
\end{remark}

The next lemma is taken from \cite[Lemma 1]{Y2004a}.

\begin{lemma}
\label{Lemma:Yamada} Let $f:[0,\infty)\rightarrow\lbrack0,\infty]$ be an
increasing function such that $f(r)=0\Leftrightarrow r=0$. Let $\{b_{k}%
\}_{k=0}^{\infty}\subseteq\lbrack0,\infty)$ be such that%
\begin{equation}
\underset{k\rightarrow\infty}{\lim}b_{k}=0\text{.}%
\end{equation}
Then any sequence $\{a_{k}\}_{k=0}^{\infty}\subseteq\lbrack0,\infty)$
satisfying%
\begin{equation}
a_{k+1}\leq a_{k}-f(a_{k})+b_{k}%
\end{equation}
converges to $0$.
\end{lemma}

Let $C\subseteq%
\mathbb{R}
^{n}$ be nonempty, closed and convex. It follows from the characterization of
the metric projection that $P_{C}$ is a cutter. Moreover, $P_{C}$ satisfies
the closedness principle as a nonexpansive operator. We now present more
examples of cutter operators the complements of which are closed at $0$.

\begin{example}
\label{Ex:cut-cl}%
\rm\
Let $U:%
\mathbb{R}
^{n}\rightarrow%
\mathbb{R}
^{n}$ be an operator with a fixed point such that $I-U$ is closed at $0$
(\textit{e.g.}, the metric projection onto a closed and convex set), and let
$T:%
\mathbb{R}
^{n}\rightarrow%
\mathbb{R}
^{n}$ be a cutter such that $\operatorname*{Fix}\left(  U\right)
\subseteq\operatorname*{Fix}\left(  T\right)  $, which satisfies for any
bounded sequence $\{x^{k}\}_{k=0}^{\infty}\subseteq%
\mathbb{R}
^{n}$ the following implication:%
\begin{equation}
\lim_{k\rightarrow\infty}\left\Vert T\left(  x^{k}\right)  -x^{k}\right\Vert
=0\Longrightarrow\lim_{k\rightarrow\infty}\left\Vert U\left(  x^{k}\right)
-x^{k}\right\Vert =0 \label{eq:Txk}%
\end{equation}
(\textit{e.g.}, a cutter $T$ such that $\left\Vert T\left(  x\right)
-x\right\Vert \geq\alpha\left\Vert U\left(  x\right)  -x\right\Vert $ for some
constant $\alpha>0$ and for all $x\in%
\mathbb{R}
^{n}$). Then $T$ is closed at $0$.\newline Indeed, let $\lim_{k\rightarrow
\infty}x^{k}=z$ and $\lim_{k\rightarrow\infty}\left\Vert T\left(
x^{k}\right)  -x^{k}\right\Vert =0$. Then $\left\{  x^{k}\right\}
_{k=0}^{\infty}$ is bounded and, by (\ref{eq:Txk}), we have $\lim
_{k\rightarrow\infty}\left\Vert U\left(  x^{k}\right)  -x^{k}\right\Vert =0$.
Since $I-U$ is closed at $0$, we know that $z\in\operatorname*{Fix}\left(
U\right)  $. Consequently, $z\in\operatorname*{Fix}\left(  T\right)  $, i.e.,
$I-T$ is closed at $0$. An operator $T$ having the properties described above
is a special case of the sequence of operators considered in \cite[Theorem
1]{Ceg07} and in \cite[Theorem 9.9]{cc11}.
\end{example}

Next we present two more examples of operators which are cutters and the
complements of which are closed at $0$. These examples are special cases of
Example \ref{Ex:cut-cl}.

First we present several definitions.\bigskip

\begin{definition}%
\rm\
\label{def:sublevel set}Let $f:%
\mathbb{R}
^{n}\rightarrow%
\mathbb{R}
$ be a convex function.

(i) The set
\begin{equation}
\partial f(y):=\{g\in\mathbb{R}^{n}\colon f(x)\geq f(y)+\langle g,x-y\rangle
\text{ for all }x\in\mathbb{R}^{n}\}
\end{equation}
is called the \texttt{subdifferential} of $f$ at $y$ and any element of
$\partial f(y)$ is called a \texttt{subgradient} and denoted by $g_{f}(y)$.

(ii) We denote by $f_{\leq0}$ the \texttt{sublevel set} of $f$, that is,
\begin{equation}
f_{\leq0}:=\left\{  x\in%
\mathbb{R}
^{n}\mid f(x)\leq0\right\}  . \label{eq:sublevel}%
\end{equation}

\end{definition}

\begin{definition}
Let $C\subseteq%
\mathbb{R}
^{n}.$ The \texttt{indicator function} of $C$ at $x$ is defined%
\begin{equation}
I_{C}(x):=\left\{
\begin{array}
[c]{ll}%
0 & \text{if\ }x\in C,\\
\infty & \text{otherwise.}%
\end{array}
\right.
\end{equation}

\end{definition}

\begin{definition}
Let $C\subseteq%
\mathbb{R}
^{n}$ be nonempty, closed and convex. We denote by $N_{C}\left(  v\right)  $
the \texttt{normal cone} of $C$ at $v\in C$, i.e.,%
\begin{equation}
N_{C}\left(  v\right)  :=\{d\in%
\mathbb{R}
^{n}\mid\left\langle d,y-v\right\rangle \leq0 \text{ for all }y\in C\}.
\label{eq:normal_cone}%
\end{equation}

\end{definition}

\begin{remark}%
\rm\
It is well known that $\partial\left(  I_{C}\right)  =N_{C}$.
\end{remark}

\begin{example}%
\rm\
\label{Exmpl. subgradients}Let $f:%
\mathbb{R}
^{n}\rightarrow%
\mathbb{R}
$ be a convex function with a nonempty sublevel set $f_{\leq0}$ (see
(\ref{eq:sublevel})). Define an operator $\Pi_{f_{\leq0}}:%
\mathbb{R}
^{n}\rightarrow%
\mathbb{R}
^{n}$ by%
\begin{equation}
\Pi_{f_{\leq0}}(y):=\left\{
\begin{array}
[c]{ll}%
y-\frac{\displaystyle f(y)}{\displaystyle\left\Vert g_{f}(y)\right\Vert ^{2}%
}g_{f}(y) & \text{if\ }f(y)>0\text{,}\\
y & \text{if\ }f(y)\leq0\text{,}%
\end{array}
\right.
\end{equation}
where $g_{f}(y)$ is a subgradient of $f$ at $y$. The operator $\Pi_{f_{\leq0}%
}$ is called the \texttt{subgradient projection}\textit{ }relative to $f$.
\end{example}

For a subgradient $g_{f}(y)$, we denote%
\begin{equation}
L=L_{f}(y,g_{f}(y)):=\{x\in%
\mathbb{R}
^{n}\mid f(y)+\left\langle g_{f}(y),x-y\right\rangle \leq0\}\text{.}
\label{Halfspase}%
\end{equation}

\begin{lemma}
\label{Lemma:subgradient projection}Let $f:%
\mathbb{R}
^{n}\rightarrow%
\mathbb{R}
$ be a convex function and let $y\in%
\mathbb{R}
^{n}$. Assume that $f_{\leq0}\neq\emptyset$. Then the following assertions hold:

\begin{itemize}
\item[$\mathrm{(i)}$] $f_{\leq0}\subseteq L$. If $g_{f}(y)\neq0,$ then $L$ is
a half-space, otherwise $L=%
\mathbb{R}
^{n}$.

\item[$\mathrm{(ii)}$] $\Pi_{f_{\leq0}}(y)=P_{L}(y)$ consequently,
$\Pi_{f_{\leq0}}$ is a cutter and $\operatorname*{Fix}\left(  \Pi_{f_{\leq0}%
}\right)  =f_{\leq0}$.

\item[$\mathrm{(iii)}$] $I-\Pi_{f_{\leq0}}$ is closed at $0$.
\end{itemize}
\end{lemma}

\begin{proof}
See \cite[Lemma 7.3]{BB}, \cite[Lemma 2.4]{cs09} and \cite[Lemma 4.2.5 and
Corollary 4.2.6]{Ceg12}.
\end{proof}

\bigskip

The next class of operators was introduced by Aharoni et. al. in \cite{abc83}
for solving the convex feasibility problem. Later Gibali \cite{Gibali} and
Censor and Gibali \cite{CG} used them for solving variational inequalities.

Let $C\subseteq%
\mathbb{R}
^{n}$ be nonempty, closed and convex. Assume that $C$ can be represented as a
sublevel set of some convex function $c:%
\mathbb{R}
^{n}\rightarrow%
\mathbb{R}
$, that is,%
\begin{equation}
C=\left\{  x\in%
\mathbb{R}
^{n}\mid c(x)\leq0\right\}  \text{.} \label{Eset}%
\end{equation}
Given a point $z\in%
\mathbb{R}
^{n}$ and a positive real number $\delta$, we define for $z\notin C$ the ball%
\begin{equation}
B(z,\delta c(z)):=\left\{  x\in%
\mathbb{R}
^{n}\mid\left\Vert x-z\right\Vert \leq\delta c(z)\right\}  \text{.}%
\end{equation}
For $x,y\in%
\mathbb{R}
^{n}$ we consider the set $H(x,y)$ as in (\ref{eq:halfspaceH}) and define%
\begin{equation}
A_{\delta}(z):=\left\{  y\in%
\mathbb{R}
^{n}\mid C\subseteq H(z,y)\text{ and }\operatorname*{int}B(z,\delta c(z))\cap
H(z,y)=\emptyset\right\}  \text{.} \label{A-delta set}%
\end{equation}
We also need to impose the following condition.

\begin{condition}%
\rm\
\label{condition delta}Given a set $C\subseteq%
\mathbb{R}
^{n}$, described as in (\ref{Eset}), we assume that for every $z\notin C,$%
\begin{equation}
B(z,\delta c(z))\cap C=\emptyset\text{.}%
\end{equation}

\end{condition}

Every convex set $C$ can be described by (\ref{Eset}). We may take, for
example, $c(z)=\operatorname*{dist}(z,C)$. In this case Condition
\ref{condition delta} always holds for $\delta\in(0,1).$

\begin{example}
\label{def. E-delta op}%
\rm\
Given a nonempty, closed and convex subset $C\subseteq%
\mathbb{R}
^{n}$, with the representation (\ref{Eset}), and a real number $\delta
\in(0,1]$ such that Condition \ref{condition delta} holds, we define the
operator $T_{C,\delta}$ at any $z\in%
\mathbb{R}
^{n}$ by%
\begin{equation}
T_{C,\delta}(z):=\left\{
\begin{array}
[c]{ll}%
P_{H(z,y)}(z) & \text{ if }z\notin C\text{,}\\
z & \text{if }z\in C\text{,}%
\end{array}
\right.  \label{e-TCd}%
\end{equation}
where $H(z,y)$ is built from any selection of $y$ from $A_{\delta}(z)$, and
call it a $C$\texttt{-}$\delta$\texttt{ operator}\textit{.}
\end{example}

Observe that the subgradient projector $\Pi_{f_{\leq0}}$ is a $T_{C,\delta}$
operator; see \cite[Lemma 2.8]{cs09}. The fact that any $C$-$\delta$ operator
is a cutter operator follows from its definition. For the closedness of
$T_{C,\delta}-I$ at $0,$ see, \textit{e.g.}, \cite[Lemma 2.7]{cs09}.
Alternatively, we show that, for any bounded sequence $\{x^{k}\}_{k=0}%
^{\infty},$ implication (\ref{eq:Txk}) is satisfied with $T:=T_{C,\delta}$.
Let $U:=\Pi_{c_{\leq0}},$ where $c:%
\mathbb{R}
^{n}\rightarrow%
\mathbb{R}
$ is a convex function with $C:=\{x\in%
\mathbb{R}
^{n}\mid c(x)\leq0\}\neq\emptyset$. Then, following Lemma
\ref{Lemma:subgradient projection}, $I-U$ is closed at $0$ and
$\operatorname*{Fix}\left(  U\right)  =C$. Now, let $\{x^{k}\}_{k=0}^{\infty}$
be bounded and let $\lim_{k\rightarrow\infty}\left\Vert T_{C,\delta}\left(
x^{k}\right)  -x^{k}\right\Vert =0$. Then, from the inequality $\left\Vert
\delta c(x^{k})\right\Vert \leq\left\Vert T_{C,\delta}\left(  x^{k}\right)
-x^{k}\right\Vert $, it follows that $\lim_{k\rightarrow\infty}c(x^{k})=0$.
Consequently, $\lim_{k\rightarrow\infty}\operatorname*{dist}(x^{k},C)=0$, by
the continuity of $c$. We claim that $\lim_{k\rightarrow\infty}\left\Vert
U\left(  x^{k}\right)  -x^{k}\right\Vert =0$. Indeed, by the definition of the
subgradient $g_{c}$ and the Cauchy--Schwarz inequality, we have%
\begin{align}
0  &  \leq c(x^{k})=c(x^{k})-c(P_{C}\left(  x^{k}\right)  )\leq\langle
g_{c}(x^{k}),x^{k}-P_{C}(x^{k})\rangle\nonumber\\
&  \leq\left\Vert g_{c}(x^{k})\right\Vert \left\Vert x^{k}-P_{C}%
(x^{k})\right\Vert =\left\Vert g_{c}(x^{k})\right\Vert \operatorname*{dist}%
(x^{k},C)
\end{align}
and%
\begin{equation}
\left\Vert U\left(  x^{k}\right)  -x^{k}\right\Vert =\frac{\left\vert
c(x^{k})\right\vert }{\left\Vert g_{c}(x^{k})\right\Vert }\leq
\operatorname*{dist}(x^{k},C)\text{.}%
\end{equation}
Consequently, $\lim_{k\rightarrow\infty}\left\Vert U\left(  x^{k}\right)
-x^{k}\right\Vert =0$, as claimed.

\section{The algorithm\label{Sec:alg}}

Let $\mathcal{F}:%
\mathbb{R}
^{n}\rightarrow%
\mathbb{R}
^{n}$ and let $T:%
\mathbb{R}
^{n}\rightarrow%
\mathbb{R}
^{n}$ be a cutter. We need to assume that the following conditions hold in
order to prove the convergence of our algorithm. These conditions were assumed
to hold in \cite{F86} for solving VIP$(\mathcal{F},S)$ (see
(\ref{eq:Classicvip})). Furthermore, the first two of these conditions
guarantee that VIP$(\mathcal{F},\operatorname*{Fix}\left(  T\right)  )$ has a
unique solution (see \cite[Theorem 2.3.3]{Facc} or \cite[Chap. I, Corollary
4.3]{KindStamp}).

\begin{condition}
\label{con:a}%
\rm\
$\mathcal{F}$ is continuous on $\left(  \operatorname*{Fix}(T)\right)
_{\varepsilon}$ for some $\varepsilon>0$.
\end{condition}

\begin{condition}
\label{con:b}%
\rm\
$\mathcal{F}$ is $\alpha$\textit{-strongly monotone} on $\left(
\operatorname*{Fix}(T)\right)  _{\varepsilon}$ for some $\varepsilon>0$ and
$\alpha>0$, i.e.,
\begin{equation}
\left\langle \mathcal{F}(x)-\mathcal{F}(y),x-y\right\rangle \geq
\alpha\left\Vert x-y\right\Vert ^{2}\text{ for all }x,y\in\left(
\operatorname*{Fix}(T)\right)  _{\varepsilon}\text{.} \label{eq:sm}%
\end{equation}

\end{condition}

\begin{condition}
\label{con:c}%
\rm\
For some $q\in\operatorname*{Fix}(T)$, there exist some $\beta>0$ and a
bounded subset $E\subseteq%
\mathbb{R}
^{n}$ such that%
\begin{equation}
\langle\mathcal{F}(x),x-q\rangle\geq\beta\Vert\mathcal{F}(x)\Vert\text{ for
all }x\notin E\text{.} \label{eq:c}%
\end{equation}

\end{condition}

\begin{condition}
\label{con:d}%
\rm\
$I-T$ is closed at $0$.\bigskip
\end{condition}

\begin{remark}
\label{remark:suffCond}%
\rm\
Conditions of the type of Condition \ref{con:c} are commonly used in
Optimization Theory (see, e.g., \cite[ Section 2.5 ]{Fletcher} and
\cite[Section 8.3]{GK99}, where also examples of methods employing these
conditions are presented). As it was observed by Fukushima in \cite{F86}, a
sufficient condition for Condition \ref{con:c} to hold is that the vectors
$\mathcal{F}(x)$ and $x$ make an acute angle, which is uniformly bounded away
from $\pi/2$, as $\Vert x\Vert\rightarrow\infty$. Indeed, Let $c\in(0,1)$ and
$r>0$ be such that
\begin{equation}
\langle\frac{\mathcal{F}(x)}{\Vert\mathcal{F}(x)\Vert},\frac{x}{\Vert x\Vert
}\rangle\geq c \label{e-c}%
\end{equation}
for $\Vert x\Vert\geq r$. Let $\beta\in(0,c)$, $q\in\operatorname{Fix}T$ and
$R\geq r$ be such that $\frac{\Vert q\Vert}{R}+\beta\leq c$. Then, for all
$\Vert x\Vert\geq R$ we obtain
\begin{align}
\langle\mathcal{F}(x),x-q\rangle &  =\left(  \left\langle \frac{\mathcal{F}%
(x)}{\Vert\mathcal{F}(x)\Vert},\frac{x}{\Vert x\Vert}\right\rangle
-\left\langle \frac{\mathcal{F}(x)}{\Vert\mathcal{F}(x)\Vert},\frac{q}{\Vert
x\Vert}\right\rangle \right)  \Vert\mathcal{F}(x)\Vert\cdot\Vert
x\Vert\nonumber\\
&  \geq\left(  c-\frac{\Vert q\Vert}{R}\right)  \Vert\mathcal{F}(x)\Vert
\cdot\Vert x\Vert\geq\beta\Vert\mathcal{F}(x)\Vert\text{.}%
\end{align}
In addition, observe that Conditions \ref{con:a} and \ref{con:b} concern the
behavior of $\mathcal{F}$on $\left(  \operatorname*{Fix}(T)\right)
_{\varepsilon}$, while Condition \ref{con:c} deals with a rather global behavior.
\end{remark}

\begin{example}%
\rm\
Let $%
\mathbb{R}
^{n}$ be equipped with the standard inner product $\langle x,y\rangle
:=x^{\intercal}y$, $x,y\in%
\mathbb{R}
^{n}$, $\lambda_{1},\lambda_{2}\in%
\mathbb{R}
$ be such that $0<\lambda_{1}\leq\lambda_{2}$. Let $c\in(0,\lambda_{1}%
/\lambda_{2})$, $a\in%
\mathbb{R}
^{n}$ and $r:=\frac{\lambda_{2}(1+c)\Vert a\Vert}{\lambda_{1}-\lambda_{2}c}$.
Define $\mathcal{F}:%
\mathbb{R}
^{n}\rightarrow%
\mathbb{R}
^{n}$ by%
\[
\mathcal{F}(x)=\left\{
\begin{array}
[c]{cc}%
\text{arbitrary} & \text{if }\Vert x\Vert\leq r\\
G(x)(x-a) & \text{if }\Vert x\Vert>r\text{,}%
\end{array}
\right.
\]
where $G(x)$ is a positive definite matrix with $\inf_{\Vert x\Vert>r}%
\lambda_{\min}(G(x))\geq\lambda_{1}>0$ and $\sup_{\Vert x\Vert>r}\lambda
_{\max}(G(x))\leq\lambda_{2}$, where $\lambda_{\min}(G(x))$ and $\lambda
_{\max}(G(x))$ denote the smallest and the largest eigenvalue of $G(x)$,
respectively. We show, that $\mathcal{F}$ satisfies (\ref{e-c}) for all $x\in%
\mathbb{R}
^{n}$ with $\Vert x\Vert>r$. It follows from the inequalities
\[
\Vert G(x)x\Vert\leq\Vert G(x)\Vert\cdot\Vert x\Vert=\lambda_{\max}(G(x))\Vert
x\Vert
\]
and%
\[
\lambda_{\min}(G(x))\Vert x\Vert^{2}\leq x^{\intercal}G(x)x\leq\lambda_{\max
}(G(x))\Vert x\Vert^{2}\text{,}%
\]
the Cauchy--Schwarz inequality, the triangle inequality and the monotonicity
of the function $\xi\rightarrow\frac{\lambda_{1}\xi-\lambda_{2}\Vert a\Vert
}{\lambda_{2}(\xi+\Vert a\Vert)}$ for $\xi>0$, that
\begin{align*}
\frac{\langle\mathcal{F}(x),x\rangle}{\Vert\mathcal{F}(x)\Vert\cdot\Vert
x\Vert}  &  =\frac{(x-a)^{\intercal}G(x)x}{\Vert G(x)(x-a)\Vert\cdot\Vert
x\Vert}=\frac{x^{\intercal}G(x)x-a^{\intercal}G(x)x}{\Vert G(x)(x-a)\Vert
\cdot\Vert x\Vert}\\
&  \geq\frac{\lambda_{\min}(G(x))\Vert x\Vert^{2}-\Vert a\Vert\lambda_{\max
}(G(x))\Vert x\Vert}{\lambda_{\max}(G(x))(\Vert x\Vert+\Vert a\Vert)\Vert
x\Vert}\\
&  \geq\frac{\lambda_{1}\Vert x\Vert-\lambda_{2}\Vert a\Vert}{\lambda
_{2}(\Vert x\Vert+\Vert a\Vert)}\geq c
\end{align*}
for all $x\in%
\mathbb{R}
^{n}$ with $\Vert x\Vert>r$. If $G(x)=G$ for all $x\in%
\mathbb{R}
^{n}$, then $\mathcal{F}(x)=G(x-a)$ and the unique solution of
VIP$(\mathcal{F},C)$ is $P_{C}^{G}a:=\operatorname{argmin}_{x\in C}\Vert
x-a\Vert_{G}$, where $\Vert\cdot\Vert_{G}$ denotes the norm induced by $G$,
i.e., $\Vert u\Vert_{G}=(u^{\intercal}Gu)^{\frac{1}{2}}$.
\end{example}

Let $\{\rho_{k}\}_{k=0}^{\infty}$ be a sequence of positive numbers satisfying%
\begin{equation}
\underset{k\rightarrow\infty}{\lim}\rho_{k}=0\text{ and }\sum_{k=1}^{\infty
}\rho_{k}=+\infty\text{.} \label{eq:ro}%
\end{equation}

\begin{algorithm}
\label{alg:ac}%
\rm\

\textbf{Initialization}: Choose an arbitrary initial point $x^{0}\in%
\mathbb{R}
^{n}$ and set $k=0$.

\textbf{Iteration step}: Given the current iterate $x^{k}$,

\begin{enumerate}
\item[(1)] build the set $H_{k}:=H(x^{k},T(x^{k}))\ $and\textbf{ }calculate
the \textquotedblleft shifted point\textquotedblright%
\begin{equation}
z^{k}:=\left\{
\begin{array}
[c]{ll}%
x^{k}-\rho_{k}\mathcal{F}(x^{k})/\Vert\mathcal{F}(x^{k})\Vert & \text{if
}\mathcal{F}(x^{k})\neq0,\\
x^{k} & \text{if }\mathcal{F}(x^{k})=0\text{.}%
\end{array}
\right.  \label{eq:z}%
\end{equation}

\item[(2)] Choose $\alpha_{k}\in\lbrack\mu,2-\mu]$ for some $\mu\in(0,1)$ and
calculate the next iterate as follows:%
\begin{equation}
x^{k+1}=P_{\alpha_{k}}(z^{k})\text{,} \label{eq:iteration}%
\end{equation}
where $P_{\alpha_{k}}=I+\alpha_{k}(P_{H_{k}}-I)$ and $P_{H_{k}}$ is the metric
projection of $%
\mathbb{R}
^{n}$ onto $H_{k}$.

\item[(3)] Set $k:=k+1$ and go to step (1).
\end{enumerate}
\end{algorithm}

\begin{remark}%
\rm\
Since $T$ is a cutter, we have $\operatorname*{Fix}\left(  T\right)  \subseteq
H_{k}$. Observe that (\ref{eq:iteration}) has an explicit form, because it is
a relaxed projection onto a half-space ($x^{k}\neq T\left(  x^{k}\right)  $):
\begin{equation}
x^{k+1}=P_{\alpha_{k}}(z^{k})=\left\{
\begin{array}
[c]{ll}%
z^{k}-\alpha_{k}\frac{\left\langle z^{k}-T\left(  x^{k}\right)  ,x^{k}%
-T\left(  x^{k}\right)  \right\rangle }{\Vert x^{k}-T\left(  x^{k}\right)
\Vert^{2}}\left(  x^{k}-T\left(  x^{k}\right)  \right)  & \text{if }%
z^{k}\notin H_{k}\text{,}\\
z^{k} & \text{if }z^{k}\in H_{k}\text{.}%
\end{array}
\right.
\end{equation}

\end{remark}

An illustration of the iterative step of Algorithm \ref{alg:ac} is given in
Figure \ref{fig:2}.%

\begin{figure}
[ptb]
\begin{center}
\includegraphics[
natheight=7.499600in,
natwidth=9.999800in,
height=2.5062in,
width=3.3347in
]%
{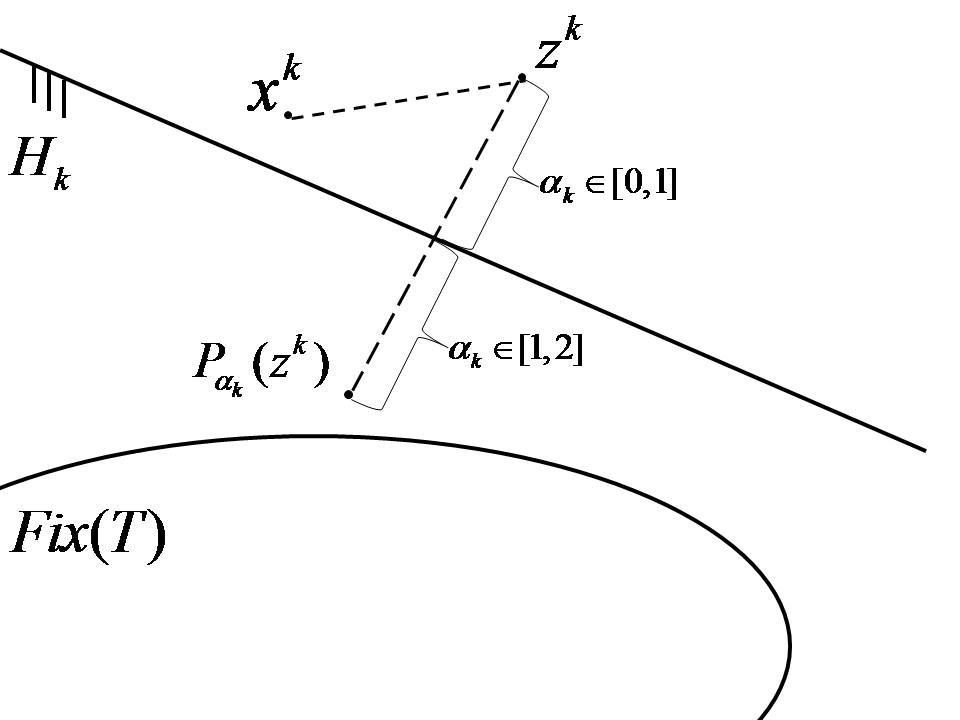}%
\caption{Iterative step of Algorithm \ref{alg:ac}}%
\label{fig:2}%
\end{center}
\end{figure}
It is clear that
\begin{equation}
\left\Vert z^{k}-x^{k}\right\Vert =\left\{
\begin{array}
[c]{ll}%
\rho_{k} & \text{ if }\mathcal{F}(x^{k})\neq0\text{,}\\
0 & \text{if }\mathcal{F}(x^{k})=0\text{.}%
\end{array}
\right.  \label{eq:inq15}%
\end{equation}
Therefore (\ref{eq:ro}) yields $\lim_{k\rightarrow\infty}\left\Vert
z^{k}-x^{k}\right\Vert =0$.

\section{Convergence\label{Sec:con}}

The following lemma is a consequence of Theorem \ref{Lem:2}, where
$P_{\alpha_{k}}$ is a relaxation of $P_{H_{k}}$, both defined in Algorithm
\ref{alg:ac}. Nevertheless, below we present a complete proof of this lemma
for the convenience of the reader.

\begin{lemma}
\label{Lem:3} Let $y\in%
\mathbb{R}
^{n}$ be arbitrary and let $\alpha_{k}\in(0,2)$. Then in the setting of
Algorithm $\mathrm{\ref{alg:ac}}$ we have
\begin{equation}
\Vert P_{\alpha_{k}}(y)-w\Vert^{2}\leq\Vert y-w\Vert^{2}-\frac{2-\alpha_{k}%
}{\alpha_{k}}\Vert P_{\alpha_{k}}(y)-y\Vert^{2}\text{\ for all\ }%
w\in\operatorname*{Fix}(T)\text{.} \label{inq:1*}%
\end{equation}
Consequently,
\begin{equation}
\operatorname*{dist}(P_{\alpha_{k}}(y),\operatorname*{Fix}(T))^{2}%
\leq\operatorname*{dist}(y,\operatorname*{Fix}(T))^{2}-\frac{2-\alpha_{k}%
}{\alpha_{k}}\Vert P_{\alpha_{k}}(y)-y\Vert^{2}\text{.} \label{eq:2*}%
\end{equation}

\end{lemma}

\begin{proof}
Let $w\in\operatorname*{Fix}(T)$. Since $\operatorname*{Fix}(T)\subseteq
H_{k}$, the characterization of the metric projection yields%

\begin{equation}
\left\langle y-P_{H_{k}}(y),y-w\right\rangle \geq\left\Vert P_{H_{k}%
}(y)-y\right\Vert ^{2}=\frac{1}{\alpha_{k}^{2}}\left\Vert P_{\alpha_{k}%
}(y)-y\right\Vert ^{2}%
\end{equation}
and therefore we have%

\begin{align}
\left\Vert P_{\alpha_{k}}(y)-w\right\Vert ^{2}  &  =\left\Vert y+\alpha
_{k}(P_{H_{k}}(y)-y)-w\right\Vert ^{2}\nonumber\\
&  =\left\Vert y-w\right\Vert ^{2}+\alpha_{k}^{2}\left\Vert P_{H_{k}%
}(y)-y\right\Vert ^{2}-2\alpha_{k}\left\langle y-P_{H_{k}}(y),y-w\right\rangle
\nonumber\\
&  \leq\left\Vert y-w\right\Vert ^{2}-\frac{2-\alpha_{k}}{\alpha_{k}%
}\left\Vert P_{\alpha_{k}}(y)-y\right\Vert ^{2}\text{.}%
\end{align}
When we set $w=P_{\operatorname*{Fix}(T)}\left(  y\right)  $ in (\ref{inq:1*}%
), we obtain
\begin{equation}
\operatorname*{dist}(P_{\alpha_{k}}(y),\operatorname*{Fix}(T))^{2}%
\leq\operatorname*{dist}(y,\operatorname*{Fix}(T))^{2}-\frac{2-\alpha_{k}%
}{\alpha_{k}}\Vert P_{\alpha_{k}}(y)-y\Vert^{2}%
\end{equation}
and the proof is complete.
\end{proof}

\begin{lemma}
\label{Lem:5} Assume that Condition $\mathrm{\ref{con:c}}$ holds. Then any
sequence $\{x^{k}\}_{k=0}^{\infty}$ generated by Algorithm
$\mathrm{\ref{alg:ac}}$ is bounded.
\end{lemma}

\begin{proof}
The proof is structured along the lines of \cite[Lemma 3]{F86}. Let
$q\in\operatorname*{Fix}(T)$, $\beta>0$ and a bounded subset $E\subseteq%
\mathbb{R}
^{n}$ be such that (\ref{eq:c}) is satisfied. We show that, for sufficiently
large $k$, we have
\begin{equation}
\Vert x^{k+1}-q\Vert\leq\Vert x^{k}-q\Vert\label{eq:inq2}%
\end{equation}
if $x^{k}\notin E$ and
\begin{equation}
\left\Vert x^{k+1}-q\right\Vert \leq\left\Vert x^{k}-q\right\Vert
+\rho\label{eq:inq2a}%
\end{equation}
otherwise, where $\rho=\sup_{k}\rho_{k}$. This implies that $\{x^{k}%
\}_{k=0}^{\infty}$ is bounded. If $\mathcal{F}(x^{k})=0$, then $x^{k}=z^{k}$.
By the definition of $x^{k+1}$ and Lemma \ref{Lem:3}, applied to $y:=z^{k}$,
$w:=q$, inequality (\ref{eq:inq2}) is satisfied. Now assume that
$\mathcal{F}(x^{k})\neq0$. Lemma \ref{Lem:3} implies that, for each $z^{k}\in%
\mathbb{R}
^{n}$,
\begin{equation}
\Vert P_{\alpha_{k}}(z^{k})-q\Vert^{2}\leq\Vert z^{k}-q\Vert^{2}\text{.}
\label{eq:l13}%
\end{equation}
Therefore%
\begin{align}
\Vert x^{k+1}-q\Vert^{2}  &  \leq\left\Vert x^{k}-\rho_{k}\frac{\mathcal{F}%
(x^{k})}{\Vert\mathcal{F}(x^{k})\Vert}-q\right\Vert ^{2}=\left\Vert
x^{k}-q\right\Vert ^{2}\nonumber\\
&  -2\frac{\rho_{k}}{\Vert\mathcal{F}(x^{k})\Vert}\langle\mathcal{F}%
(x^{k}),x^{k}-q\rangle+\rho_{k}^{2}\text{.} \label{eq:inq}%
\end{align}
Thus, if $x^{k}\notin E$, we have, by (\ref{eq:inq}) and Condition
\ref{con:c},%
\begin{equation}
\Vert x^{k+1}-q\Vert^{2}\leq\Vert x^{k}-q\Vert^{2}-2\rho_{k}\beta+\rho_{k}%
^{2}=\Vert x^{k}-q\Vert^{2}-\rho_{k}(2\beta-\rho_{k})\text{.} \label{eq:inq9}%
\end{equation}
Since $\rho_{k}>0$ and $\lim_{k\rightarrow\infty}\rho_{k}=0$, the latter
inequality implies (\ref{eq:inq2}), provided that $k$ is sufficiently large.
On the other hand, by (\ref{eq:l13}), the definition of $z^{k}$ (see
(\ref{eq:z})) and the triangle inequality, we obtain
\begin{equation}
\left\Vert x^{k+1}-q\right\Vert \leq\left\Vert (x^{k}-q)-\rho_{k}%
\frac{\mathcal{F}(x^{k})}{\Vert\mathcal{F}(x^{k})\Vert}\right\Vert
\leq\left\Vert x^{k}-q\right\Vert +\rho_{k}\text{,}%
\end{equation}
i.e., (\ref{eq:inq2a}) is satisfied. Hence $\{x^{k}\}_{k=0}^{\infty}$ is
indeed bounded as we have already observed.
\end{proof}

\begin{lemma}
\label{Lem:8} Assume that Condition $\mathrm{\ref{con:c}}$ holds. Then any
sequence $\{x^{k}\}_{k=0}^{\infty}$ generated by Algorithm
$\mathrm{\ref{alg:ac}}$ satisfies%
\begin{equation}
\lim_{k\rightarrow\infty}\operatorname*{dist}(x^{k},\operatorname*{Fix}%
(T))=0\text{.}%
\end{equation}

\end{lemma}

\begin{proof}
Recall that $R:=P_{\operatorname{Fix}(T)}$. Lemma \ref{Lem:3} implies that
\begin{equation}
\operatorname*{dist}(x^{k+1},\operatorname*{Fix}(T))\leq\operatorname*{dist}%
(z^{k},\operatorname*{Fix}(T))\text{.}%
\end{equation}
If $\mathcal{F}(x^{k})\neq0$, then, by the triangle inequality,%
\begin{align}
\operatorname*{dist}(z^{k},\operatorname*{Fix}(T))  &  =\Vert z^{k}%
-R(x^{k})\Vert=\left\Vert x^{k}-\rho_{k}\frac{\mathcal{F}(x^{k})}%
{\Vert\mathcal{F}(x^{k})\Vert}-R(x^{k})\right\Vert \nonumber\\
&  \leq\left\Vert x^{k}-R(x^{k})\right\Vert +\rho_{k}\nonumber\\
&  =\operatorname*{dist}(x^{k},\operatorname*{Fix}(T))+\rho_{k}\text{.}%
\end{align}
If $\mathcal{F}(x^{k})=0$, then $x^{k}=z^{k}$ and, consequently,
\begin{equation}
\operatorname*{dist}(z^{k},\operatorname*{Fix}(T))=\operatorname*{dist}%
(x^{k},\operatorname*{Fix}(T)).
\end{equation}
Therefore in both cases we have%
\begin{equation}
\operatorname*{dist}(x^{k+1},\operatorname*{Fix}(T))\leq\operatorname*{dist}%
(x^{k},\operatorname*{Fix}(T))+\widetilde{\rho}_{k}\text{,}%
\end{equation}
where
\begin{equation}
\widetilde{\rho}_{k}:=%
\begin{cases}
\rho_{k}, & \text{if }\mathcal{F}(x^{k})\neq0,\\
0, & \text{if }\mathcal{F}(x^{k})=0.
\end{cases}
\label{e-ro-tilda}%
\end{equation}
Define $a_{k}:=\operatorname{dist}(x^{k},\operatorname*{Fix}(T))$. Let
$C\subseteq%
\mathbb{R}
^{n}$ be a bounded, closed and convex set such that $\{x^{k}\}_{k=0}^{\infty
}\subseteq C$ and $\operatorname*{Fix}(T)\cap C\neq\emptyset$. The existence
of such a set $C$ follows from Lemma \ref{Lem:5}. By Proposition
\ref{proposition:D}(iii)and Lemma \ref{Lem:3},
\begin{align}
D^{2}(a_{k})  &  =D^{2}(\operatorname*{dist}(x^{k},\operatorname*{Fix}%
(T)))\leq\Vert x^{k}-T(x^{k})\Vert^{2}=\frac{1}{\alpha_{k}^{2}}\Vert
x^{k}-P_{\alpha_{k}}(x^{k})\Vert^{2}\\
&  \leq\frac{1}{\alpha_{k}\left(  2-\alpha_{k}\right)  }\left(  \left(
\operatorname*{dist}(x^{k},\operatorname*{Fix}(T))\right)  ^{2}-\left(
\operatorname*{dist}(P_{\alpha_{k}}\left(  x^{k}\right)  ,\operatorname*{Fix}%
(T))\right)  ^{2}\right)  \text{.} \label{eq:dwave}%
\end{align}
On the other hand, by the nonexpansivity of $P_{\alpha_{k}}$, (\ref{eq:inq15})
and (\ref{e-ro-tilda}), we get%
\begin{equation}
\Vert x^{k+1}-P_{\alpha_{k}}(x^{k})\Vert^{2}=\Vert P_{\alpha_{k}}%
(z^{k})-P_{\alpha_{k}}(x^{k})\Vert^{2}\leq\Vert z^{k}-x^{k}\Vert
^{2}=\widetilde{\rho}_{k}^{\ 2}\text{.}%
\end{equation}
Therefore%
\begin{equation}
\Vert x^{k+1}-P_{\alpha_{k}}(x^{k})\Vert\leq\widetilde{\rho}_{k}\text{.}%
\end{equation}
Let $s^{k}=R\left(  P_{\alpha_{k}}\left(  x^{k}\right)  \right)  $, i.e.,%
\begin{equation}
\Vert P_{\alpha_{k}}(x^{k})-s^{k}\Vert=\operatorname*{dist}(P_{\alpha_{k}%
}(x^{k}),\operatorname*{Fix}(T))\text{.}%
\end{equation}
Then, by the triangle inequality, we get%
\begin{equation}
\Vert x^{k+1}-s^{k}\Vert\leq\Vert x^{k+1}-P_{\alpha_{k}}(x^{k})\Vert+\Vert
P_{\alpha_{k}}(x^{k})-s^{k}\Vert\text{.}%
\end{equation}
On the other hand, since $s^{k}\in\operatorname*{Fix}(T)$, we have%
\begin{equation}
\operatorname*{dist}(x^{k+1},\operatorname*{Fix}(T))\leq\Vert x^{k+1}%
-s^{k}\Vert\text{.}%
\end{equation}
From the last four inequalities we get
\begin{align}
a_{k+1}  &  \leq\Vert x^{k+1}-P_{\alpha_{k}}(x^{k})\Vert+\operatorname*{dist}%
(P_{\alpha_{k}}(x^{k}),\operatorname*{Fix}(T))\nonumber\\
&  \leq\widetilde{\rho}_{k}+\operatorname*{dist}(P_{\alpha_{k}}(x^{k}%
),\operatorname*{Fix}(T))
\end{align}
or, equivalently,%
\begin{equation}
\left(  \operatorname*{dist}(P_{\alpha_{k}}(x^{k}),\operatorname*{Fix}%
(T))\right)  ^{2}\geq a_{k+1}^{2}-\widetilde{\rho}_{k}^{\ 2}-2\widetilde{\rho
}_{k}\operatorname*{dist}(P_{\alpha_{k}}(x^{k}),\operatorname*{Fix}%
(T))\text{.}%
\end{equation}
Using the above inequality for (\ref{eq:dwave}), we get for all $k\geq0,$%
\begin{equation}
D^{2}(a_{k})\leq\frac{1}{\alpha_{k}\left(  2-\alpha_{k}\right)  }\left(
a_{k}^{2}-a_{k+1}^{2}+\widetilde{\rho}_{k}^{\ 2}+2\widetilde{\rho}%
_{k}\operatorname*{dist}(P_{\alpha_{k}}(x^{k}),\operatorname*{Fix}(T))\right)
\text{.} \label{eq:panch}%
\end{equation}
Now, by Lemma \ref{Lem:5}, the sequence $\{x^{k}\}_{k=0}^{\infty}$ is bounded
and, therefore, so is $\{a_{k}\}_{k=0}^{\infty}$. By Lemma \ref{Lem:3} with
$y:=x^{k}$, we have%
\begin{align}
&  \left(  \operatorname*{dist}(P_{\alpha_{k}}(x^{k}),\operatorname*{Fix}%
(T))\right)  ^{2}\nonumber\\
&  \leq\Vert P_{\alpha_{k}}(x^{k})-z\Vert^{2}\leq\Vert x^{k}-z\Vert^{2}%
-\frac{2-\alpha_{k}}{\alpha_{k}}\Vert P_{\alpha_{k}}(x^{k})-x^{k}\Vert^{2}%
\leq\Vert x^{k}-z\Vert^{2}%
\end{align}
for all $z\in\operatorname*{Fix}(T)$. Taking $z:=R\left(  x^{k}\right)  $\ in
the above inequalities, we obtain\textit{ }%
\begin{equation}
\operatorname*{dist}(P_{\alpha_{k}}(x^{k}),\operatorname*{Fix}(T))\leq
a_{k}=\operatorname*{dist}(x^{k},\operatorname*{Fix}(T))\text{.}%
\end{equation}
Therefore the sequence $\{\operatorname*{dist}(P_{\alpha_{k}}(x^{k}%
),\operatorname*{Fix}(T))\}_{k=0}^{\infty}$ is also bounded. Since $\alpha
_{k}\in\lbrack\mu,2-\mu]$ for some $\mu\in(0,1)$, we have $1/\left(
\alpha_{k}\left(  2-\alpha_{k}\right)  \right)  \leq1/\mu^{2}$. Denote
$b_{k}:=\widetilde{\rho}_{k}^{\ 2}+2\widetilde{\rho}_{k}\operatorname*{dist}%
(P_{\alpha_{k}}(x^{k}),\operatorname*{Fix}(T))$. Using (\ref{eq:panch}), we
get%
\begin{align}
D^{2}(a_{k})  &  \leq\frac{1}{\mu^{2}}\left(  a_{k}^{2}-a_{k+1}^{2}%
+b_{k}\right) \nonumber\\
&  =\frac{1}{\mu^{2}}\left(  a_{k}+a_{k+1}\right)  \left(  a_{k}%
-a_{k+1}\right)  +\frac{1}{\mu^{2}}b_{k}\text{.}%
\end{align}
Since $\{a_{k}\}_{k=0}^{\infty}$ is bounded , we have $1/\mu^{2}\left(
a_{k}+a_{k+1}\right)  \leq M$ for some $M>0$. By the definition of
$\widetilde{\rho}_{k}$, $\lim_{k\rightarrow\infty}b_{k}=0$. Hence%
\begin{equation}
\frac{1}{M}D^{2}(a_{k})\leq a_{k}-a_{k+1}+\frac{1}{\mu^{2}M}b_{k}%
\end{equation}
and now we can apply Lemma \ref{Lemma:Yamada} to conclude that $\lim
_{k\rightarrow\infty}\operatorname*{dist}(x^{k},\operatorname*{Fix}(T))=0,$
which completes the proof.
\end{proof}

\begin{lemma}
\label{Lem:10} Assume that Condition $\mathrm{\ref{con:c}}$ holds. Any
sequence $\{x^{k}\}_{k=0}^{\infty}$ generated by Algorithm
$\mathrm{\ref{alg:ac}}$ satisfies%
\begin{equation}
\lim_{k\rightarrow\infty}\Vert x^{k+1}-x^{k}\Vert=0\text{.}%
\end{equation}

\end{lemma}

\begin{proof}
If $\mathcal{F}(x^{k})\neq0$, then by the triangle inequality, the
nonexpansivity of $P_{\alpha_{k}}$ and by (\ref{eq:inq15}), we obtain for all
$k\geq0,$%
\begin{align}
\Vert x^{k+1}-x^{k}\Vert &  \leq\Vert x^{k+1}-P_{\alpha_{k}}(x^{k})\Vert+\Vert
P_{\alpha_{k}}(x^{k})-x^{k}\Vert\nonumber\\
&  =\Vert P_{\alpha_{k}}(z^{k})-P_{\alpha_{k}}(x^{k})\Vert+\Vert P_{\alpha
_{k}}(x^{k})-x^{k}\Vert\nonumber\\
&  \leq\Vert z^{k}-x^{k}\Vert+\Vert P_{\alpha_{k}}(x^{k})-x^{k}\Vert
\nonumber\\
&  \leq\rho_{k}+\alpha_{k}\operatorname*{dist}(x^{k},H_{k})\text{,}
\label{eq:xk1}%
\end{align}
where the latter inequality follows from (\ref{eq:inq15}) and from the obvious
equality $\alpha_{k}\operatorname*{dist}(x^{k},H_{k})=\Vert P_{\alpha_{k}%
}(x^{k})-x^{k}\Vert$. Since for all $k\geq0$,$\ \operatorname*{Fix}%
(T)\subseteq H_{k}$, we have%
\begin{equation}
\operatorname*{dist}(x^{k},H_{k})\leq\operatorname*{dist}(x^{k}%
,\operatorname*{Fix}(T))\text{.}%
\end{equation}
Thus,%
\begin{equation}
\Vert x^{k+1}-x^{k}\Vert\leq\rho_{k}+\alpha_{k}\operatorname*{dist}%
(x^{k},\operatorname*{Fix}(T))\text{.}%
\end{equation}

In case $\mathcal{F}(x^{k})=0,$ we have%
\begin{equation}
\Vert x^{k+1}-x^{k}\Vert=\alpha_{k}\operatorname*{dist}(x^{k},H_{k})\leq
\alpha_{k}\operatorname*{dist}(x^{k},\operatorname*{Fix}(T))\text{.}%
\end{equation}
By Lemma \ref{Lem:8} and (\ref{eq:ro}), we obtain the required result.
\end{proof}

\begin{theorem}
\label{Theorem:1} Assume that Conditions $\mathrm{\ref{con:a}-\ref{con:d}}$
are satisfied. Then any sequence $\{x^{k}\}_{k=0}^{\infty}$ generated by
Algorithm $\mathrm{\ref{alg:ac}}$ converges to the unique solution $x^{\ast}$
of VIP$(\mathcal{F},\operatorname*{Fix}\left(  T\right)  )$.
\end{theorem}

\begin{proof}
Let $x^{\ast}$ be the unique solution of problem (\ref{eq:vip}). By Lemma
\ref{Lem:8}, $x^{k}\in(\operatorname*{Fix}(T))_{\varepsilon}$ for all
sufficiently large $k$, where $(\operatorname*{Fix}(T))_{\varepsilon}$ is the
set given in Conditions \ref{con:a} and \ref{con:b} (without loss of
generality, $\varepsilon$ is the same in both conditions). By Condition
\ref{con:b}, we have%
\begin{equation}
\langle\mathcal{F}(x^{k})-\mathcal{F}(x^{\ast}),x^{k}-x^{\ast}\rangle
\geq\alpha\Vert x^{k}-x^{\ast}\Vert^{2}%
\end{equation}
and%
\begin{align}
\langle\mathcal{F}(x^{k})-\mathcal{F}(x^{\ast}),x^{k}-x^{\ast}\rangle &
=\langle\mathcal{F}(x^{k}),x^{k}-x^{k+1}\rangle+\langle\mathcal{F}%
(x^{k}),x^{k+1}-x^{\ast}\rangle\nonumber\\
&  -\langle\mathcal{F}(x^{\ast}),x^{k}-x^{\ast}\rangle\text{.}%
\end{align}
Therefore%
\begin{equation}
\langle\mathcal{F}(x^{k}),x^{k+1}-x^{\ast}\rangle\geq\alpha\Vert x^{k}%
-x^{\ast}\Vert^{2}+\langle\mathcal{F}(x^{\ast}),x^{k}-x^{\ast}\rangle
+\langle\mathcal{F}(x^{k}),x^{k+1}-x^{k}\rangle\text{.} \label{eq:inq24}%
\end{equation}
Let $\lambda$ be an arbitrary positive number. Then we have%
\begin{equation}
\langle\mathcal{F}(x^{\ast}),x^{k}-x^{\ast}\rangle\geq-\lambda
\label{eq:lamda1}%
\end{equation}
for all sufficiently large $k$. Indeed. The inequality is clear if
$\mathcal{F}(x^{\ast})=0$. Assume now that $\mathcal{F}(x^{\ast})\neq0$ and
let $\varepsilon:=\lambda/\Vert\mathcal{F}(x^{\ast})\Vert$. By Lemma
\ref{Lem:8}, we have $\Vert x^{k}-z^{k}\Vert\leq\varepsilon$ for all
sufficiently large $k$, where $z^{k}=P_{\operatorname*{Fix}(T)}(x^{k})$.
Consequently,
\begin{align*}
\langle\mathcal{F}(x^{\ast}),x^{k}-x^{\ast}\rangle &  =\langle\mathcal{F}%
(x^{\ast}),z^{k}-x^{\ast}\rangle+\langle\mathcal{F}(x^{\ast}),x^{k}%
-z^{k}\rangle\\
&  \geq-\Vert\mathcal{F}(x^{\ast})\Vert\cdot\Vert x^{k}-z^{k}\Vert\geq
-\Vert\mathcal{F}(x^{\ast})\Vert\cdot\varepsilon=-\lambda\text{.}%
\end{align*}

By the Cauchy--Schwarz inequality,%
\begin{equation}
\langle\mathcal{F}(x^{k}),x^{k+1}-x^{k}\rangle\geq-\Vert\mathcal{F}%
(x^{k})\Vert\Vert x^{k+1}-x^{k}\Vert\text{.} \label{eq:Fxk}%
\end{equation}
From the boundedness of $\{x^{k}\}_{k=0}^{\infty}$ (see Lemma \ref{Lem:5}) and
from Lemma \ref{Lem:8}, it follows, due to the continuity of $\mathcal{F}$ on
$(\operatorname*{Fix}(T))_{\varepsilon}$ (Condition \ref{con:a}), that the
sequence $\{\mathcal{F}(x^{k})\}_{k=0}^{\infty}$ is also bounded. Lemma
\ref{Lem:10} and inequality (\ref{eq:Fxk}) guarantee that%
\begin{equation}
\langle\mathcal{F}(x^{k}),x^{k+1}-x^{k}\rangle\geq-\lambda\label{eq:lamda2}%
\end{equation}
for all sufficiently large $k$. Applying (\ref{eq:lamda1}) and
(\ref{eq:lamda2}) to (\ref{eq:inq24}), we obtain%
\begin{equation}
\langle\mathcal{F}(x^{k}),x^{k+1}-x^{\ast}\rangle\geq\alpha\Vert x^{k}%
-x^{\ast}\Vert^{2}-2\lambda\label{eq:inq5}%
\end{equation}
for all sufficiently large $k$. Divide the indices of $\{x^{k}\}_{k=0}%
^{\infty}$ as follows:%
\begin{equation}
\Gamma:=\{k\geq0\mid\mathcal{F}(x^{k})=0\}\text{ and }\tilde{\Gamma}%
:=\{k\geq0\mid\mathcal{F}(x^{k})\neq0\}\text{.}%
\end{equation}
Equation (\ref{eq:inq5}) implies, since $\lambda$ is arbitrary, that for
$k\in\Gamma,$%
\begin{equation}
\lim_{k\rightarrow\infty}x^{k}=x^{\ast}\text{.}%
\end{equation}
We now show that the sequence $\{x^{k}\}_{k\in\tilde{\Gamma}}$ contains a
subsequence which converges to $x^{\ast}$. To this end, let us consider the
indices in $\tilde{\Gamma}$ and suppose that there exist $\zeta>0$ and an
integer $k_{0}$ such that%
\begin{equation}
\Vert x^{k}-x^{\ast}\Vert\geq\zeta\ \text{for all }k\in\tilde{\Gamma},k\geq
k_{0}. \label{eq:inq6}%
\end{equation}
By Lemma \ref{Lem:3},%
\begin{align}
\Vert x^{k+1}-x^{\ast}\Vert^{2}  &  =\Vert P_{\alpha_{k}}(z^{k})-x^{\ast}%
\Vert^{2}\leq\Vert z^{k}-x^{\ast}\Vert^{2}-\frac{2-\alpha_{k}}{\alpha_{k}%
}\Vert P_{\alpha_{k}}(z^{k})-z^{k}\Vert^{2}\nonumber\\
&  =\left\Vert \left(  x^{k}-\rho_{k}\frac{\mathcal{F}(x^{k})}{\Vert
\mathcal{F}(x^{k})\Vert}\right)  -x^{\ast}\right\Vert ^{2}\nonumber\\
&  -\frac{2-\alpha_{k}}{\alpha_{k}}\left\Vert x^{k+1}-\left(  x^{k}-\rho
_{k}\frac{\mathcal{F}(x^{k})}{\Vert\mathcal{F}(x^{k})\Vert}\right)
\right\Vert ^{2}\nonumber\\
&  =\Vert x^{k}-x^{\ast}\Vert^{2}-\frac{2-\alpha_{k}}{\alpha_{k}}\Vert
x^{k+1}-x^{k}\Vert^{2}-2\frac{\rho_{k}}{\Vert\mathcal{F}(x^{k})\Vert}%
\langle\mathcal{F}(x^{k}),x^{k}-x^{\ast}\rangle\nonumber\\
&  +2\frac{\left(  2-\alpha_{k}\right)  \rho_{k}}{\alpha_{k}\Vert
\mathcal{F}(x^{k})\Vert}\langle\mathcal{F}(x^{k}),x^{k+1}-x^{k}\rangle
+\rho_{k}^{2}\left(  1-\frac{2-\alpha_{k}}{\alpha_{k}}\right)  .
\end{align}
This is less than or equal to%
\begin{align}
&  \Vert x^{k}-x^{\ast}\Vert^{2}-2\frac{\rho_{k}}{\Vert\mathcal{F}(x^{k}%
)\Vert}\langle\mathcal{F}(x^{k}),x^{k}-x^{\ast}\rangle\nonumber\\
&  +2\frac{\left(  2-\alpha_{k}\right)  \rho_{k}}{\alpha_{k}\Vert
\mathcal{F}(x^{k})\Vert}\langle\mathcal{F}(x^{k}),x^{k+1}-x^{k}\rangle
+\rho_{k}^{2}\left(  1-\frac{2-\alpha_{k}}{\alpha_{k}}\right) \nonumber\\
&  =\Vert x^{k}-x^{\ast}\Vert^{2}-2\frac{\rho_{k}}{\Vert\mathcal{F}%
(x^{k})\Vert}\langle\mathcal{F}(x^{k}),x^{k+1}-x^{\ast}\rangle\nonumber\\
&  +2\left(  \frac{2-\alpha_{k}}{\alpha_{k}}+1\right)  \frac{\rho_{k}}%
{\Vert\mathcal{F}(x^{k})\Vert}\langle\mathcal{F}(x^{k}),x^{k+1}-x^{k}%
\rangle\nonumber\\
&  +\rho_{k}^{2}\left(  1-\frac{2-\alpha_{k}}{\alpha_{k}}\right)  .
\end{align}
So, combining the above relations, we obtain%
\begin{align}
\Vert x^{k+1}-x^{\ast}\Vert^{2}  &  \leq\Vert x^{k}-x^{\ast}\Vert^{2}%
-2\frac{\rho_{k}}{\Vert\mathcal{F}(x^{k})\Vert}\langle\mathcal{F}%
(x^{k}),x^{k+1}-x^{\ast}\rangle\nonumber\\
&  +2\left(  \frac{2-\alpha_{k}}{\alpha_{k}}+1\right)  \frac{\rho_{k}}%
{\Vert\mathcal{F}(x^{k})\Vert}\langle\mathcal{F}(x^{k}),x^{k+1}-x^{k}%
\rangle\nonumber\\
&  +\rho_{k}^{2}\left(  1-\frac{2-\alpha_{k}}{\alpha_{k}}\right)  \text{.}
\label{eq:inq:xKxStar1}%
\end{align}
Since $\{\mathcal{F}(x^{k})\}_{k=0}^{\infty}$ is bounded, there exists
$\tau>0$ such that $\Vert\mathcal{F}(x^{k})\Vert\leq\tau$. Therefore,
\begin{equation}
-\frac{1}{\Vert\mathcal{F}(x^{k})\Vert}\leq-\frac{1}{\tau},\text{for all
$k\in$}\tilde{\Gamma}\text{.} \label{eq:inq:xKxStar2}%
\end{equation}
Since $\lambda$ is arbitrary, we can assume that
\begin{equation}
2\lambda\leq\frac{1}{4}\alpha\zeta^{2}.
\end{equation}
By similar arguments as in derivation of inequality (\ref{eq:lamda2}) and by
the boundedness of the sequence $\{\left(  2-\alpha_{k}\right)  /\alpha
_{k}\}_{k=0}^{\infty}$, we can assume that, for all sufficiently large $k$,%
\begin{equation}
\left(  \frac{2-\alpha_{k}}{\alpha_{k}}+1\right)  \langle\mathcal{F}%
(x^{k}),x^{k+1}-x^{k}\rangle\leq\frac{1}{4}\alpha\zeta^{2}-2\lambda\text{.}
\label{eq:inq:xKxStar3}%
\end{equation}
Since $\rho_{k}\rightarrow0$ and again using the boundedness of $\left(
2-\alpha_{k}\right)  /\alpha_{k}$, we can also assume that, for all
sufficiently large $k$,
\begin{equation}
\rho_{k}\left(  1-\frac{2-\alpha_{k}}{\alpha_{k}}\right)  \leq\frac
{\alpha\zeta^{2}}{2\tau}\text{.} \label{eq:inq:xKxStar4}%
\end{equation}
Applying (\ref{eq:inq5}), (\ref{eq:inq:xKxStar3}) and (\ref{eq:inq:xKxStar4})
to (\ref{eq:inq:xKxStar1}), we get
\begin{align}
\Vert x^{k+1}-x^{\ast}\Vert^{2}  &  \leq\Vert x^{k}-x^{\ast}\Vert^{2}%
-2\frac{\rho_{k}}{\Vert\mathcal{F}(x^{k})\Vert}(\alpha\zeta^{2}-2\lambda
)\nonumber\\
&  +2\frac{\rho_{k}}{\Vert\mathcal{F}(x^{k})\Vert}\left(  \frac{1}{4}%
\alpha\zeta^{2}-2\lambda\right)  +\rho_{k}\frac{\alpha\zeta^{2}}{2\tau
}\nonumber\\
&  =\Vert x^{k}-x^{\ast}\Vert^{2}-\frac{3}{2}\frac{\rho_{k}}{\Vert
\mathcal{F}(x^{k})\Vert}\alpha\zeta^{2}+\rho_{k}\frac{\alpha\zeta^{2}}{2\tau
}\text{.} \label{eq:inq:xKxStar5}%
\end{align}
Combining (\ref{eq:inq:xKxStar2}) with (\ref{eq:inq:xKxStar5}), we get
\begin{equation}
\Vert x^{k+1}-x^{\ast}\Vert^{2}\leq\Vert x^{k}-x^{\ast}\Vert^{2}-\rho_{k}%
\frac{\alpha\zeta^{2}}{\tau}%
\end{equation}
for all sufficiently large $k$. Thus there exists an integer $\tilde{k}%
\in\tilde{\Gamma}$ such that%
\begin{equation}
\Vert x^{k+1}-x^{\ast}\Vert^{2}\leq\Vert x^{k}-x^{\ast}\Vert^{2}-\rho_{k}%
\frac{\alpha\zeta^{2}}{\tau}\text{\ for all }k\in\tilde{\Gamma}\text{ such
that }k\geq\tilde{k}\text{.} \label{eq:inq7}%
\end{equation}
By adding these inequalities for $k=\tilde{k},\tilde{k}+1,...$,$\tilde{k}%
+\ell$ over $k\in\tilde{\Gamma}$, we obtain%
\begin{equation}
\Vert x^{\tilde{k}+\ell+1}-x^{\ast}\Vert^{2}\leq\Vert x\tilde{^{k}}-x^{\ast
}\Vert^{2}-\frac{\alpha\zeta^{2}}{\tau}\sum_{k\in\tilde{\Gamma},k=\tilde{k}%
}^{\tilde{k}+\ell}\rho_{k}%
\end{equation}
for any $\ell>0$. However, this is impossible in view of (\ref{eq:ro}). So
there exists no $\zeta>0$ such that (\ref{eq:inq6}) is satisfied. Therefore
$\{x^{k}\}_{k\in\tilde{\Gamma}}$ contains a subsequence $\{x^{k}\}_{k\in
\hat{\Gamma}}$, $\hat{\Gamma}\subseteq$ $\tilde{\Gamma},$ converging to
$x^{\ast}$, i.e., there is a subsequence $\{x^{k}\}_{k\in\Gamma\cup\hat
{\Gamma}}$ of the whole sequence $\{x^{k}\}_{k=0}^{\infty}$ which converges to
$x^{\ast}$. In order to prove that the entire sequence $\{x^{k}\}_{k=0}%
^{\infty}$ converges to $x^{\ast}$, suppose to the contrary that there exists
a subsequence of $\{x^{k}\}_{k=0}^{\infty}$ converging to $\hat{x}$ and
$\hat{x}\neq x^{\ast}$. By Lemma \ref{Lem:10}, $\lim_{k\rightarrow\infty}\Vert
x^{k+1}-x^{k}\Vert=0$, therefore, there exists $\zeta>0$ and an arbitrarily
large integer $j\in\tilde{\Gamma}$ such that%
\begin{equation}
\Vert x^{j}-x^{\ast}\Vert\geq\zeta\text{\ and\ }\Vert x^{j+1}-x^{\ast}%
\Vert\geq\Vert x^{j}-x^{\ast}\Vert\text{.} \label{eq:inq8}%
\end{equation}
However, if $j$ is sufficiently large, we may apply an argument similar to
that used to derive (\ref{eq:inq7}) and obtain the inequality%
\begin{equation}
\Vert x^{j+1}-x^{\ast}\Vert<\Vert x^{j}-x^{\ast}\Vert\text{,}%
\end{equation}
which contradicts (\ref{eq:inq8}). Therefore the whole sequence $\{x^{k}%
\}_{k=0}^{\infty}$ does converge to $x^{\ast},$ as asserted.
\end{proof}

\begin{remark}%
\rm\
In \cite{Y2004a} the operator $\mathcal{F}$ is assumed to be Lipschitz
continuous and strongly monotone on the image of $T$, while here $\mathcal{F}$
is only assumed to be continuous on $\left(  \operatorname*{Fix}(T)\right)
_{\varepsilon}$ for some $\varepsilon>0$. In \cite[Theorem 5]{Y2004a} Yamada
and Ogura showed that the strong monotonicity of $\mathcal{F}$ could be
weakened and replaced by the paramonotonicity. In \cite{yyy} Yamada et al.
applied successfully Algorithm (\ref{eq:yamada}) to the minimization of the
Moreau envelope of nonsmooth convex functions, where only Lipschitz continuity
and paramonotonicity were assumed.
\end{remark}

\section{An application\label{Sec:app}}

Given an operator $f:%
\mathbb{R}
^{n}\rightarrow%
\mathbb{R}
^{m}$, we would like to find its minimizers. Clearly, we cannot look for an
optimal solution as defined for a scalar optimization problem ($m=1$).
Therefore we need to define \textit{a priori} which solution concept is
chosen. One might consider the lexicographic order, denoted by $\preceq_{L}$.
This partial order is defined for $x,y\in%
\mathbb{R}
^{m}\ $as follows:%
\begin{equation}
x\preceq_{L}y\Leftrightarrow x=y\text{ or }x_{k}<y_{k}\text{ where }%
k:=\min\{i=1,\ldots,m\mid x_{i}\neq y_{i}\}\text{.}%
\end{equation}
Now consider the case where $m=2$, i.e., $f:%
\mathbb{R}
^{n}\rightarrow%
\mathbb{R}
^{2}$, and denote by $f_{i}:%
\mathbb{R}
^{n}\rightarrow%
\mathbb{R}
$ the $i$-th coordinate ($i=1,2$) of the function $f$. Then our goal is to
minimize $f$ with respect to $\preceq_{L}$. This problem is also called a
\textit{two-stage} or a \textit{bi-level} or a \textit{hierarchical
optimization problem}. Before introducing the connection of this problem to
our VIP, we recall some definitions and properties.

\begin{definition}%
\rm\
Let $A:%
\mathbb{R}
^{n}\rightarrow2^{%
\mathbb{R}
^{n}}\mathcal{\ }$be a set-valued mapping.

(i) $A$ is called a \texttt{maximal monotone mapping} if it is
\texttt{monotone}, i.e.,%
\begin{equation}
\left\langle u-v,x-y\right\rangle \geq0\text{ for all }u\in A(x)\text{ and
}v\in A(y)\text{,}%
\end{equation}
and the \texttt{graph} $G(A)$ of $A$,%
\begin{equation}
G(A):=\left\{  \left(  x,u\right)  \in%
\mathbb{R}
^{n}\times%
\mathbb{R}
^{n}\mid u\in A(x)\right\}  \text{,}%
\end{equation}
is not properly contained in the graph of any other monotone mapping.

(ii) The \texttt{resolvent} of $A$ with parameter $\lambda$ is the operator
$J_{\lambda}^{A}:=\left(  I+\lambda A\right)  ^{-1}$, where $I$ is the
identity operator.
\end{definition}

\begin{remark}%
\rm\
\label{remark:resolvent}It is well known that for $\lambda>0$,

(i) $A$ is monotone if and only if the resolvent $J_{\lambda}^{A}$ of $A$ is
single-valued and firmly nonexpansive.

(ii) $A$ is maximal monotone if and only if $J_{\lambda}^{A}$ is
single-valued, firmly nonexpansive and its domain is $%
\mathbb{R}
^{n}$, where%
\begin{equation}
\operatorname*{dom}(J_{\lambda}^{A}):=\left\{  x\in%
\mathbb{R}
^{n}\mid J_{\lambda}^{A}(x)\neq\emptyset\right\}  \text{.}%
\end{equation}

(iii)%
\begin{equation}
0\in A(x)\Leftrightarrow x\in\operatorname{Fix}(J_{\lambda}^{A})\text{.}
\label{eq:res-fix}%
\end{equation}

\end{remark}

\begin{example}%
\rm\
Let $C\subseteq%
\mathbb{R}
^{n}$ be nonempty, closed and convex. The metric projection onto $C$ is
precisely the resolvent of the normal cone mapping, \textit{i.e.},
\begin{equation}
P_{C}=J_{\lambda}^{N_{C}}.
\end{equation}
In addition, it is known that $N_{C}$ is a maximal monotone mapping.
\end{example}

\begin{remark}%
\rm\
Let $C\subseteq\mathbb{R}^{n}$ be closed and convex. If a function
$g:\mathbb{R}^{n}\rightarrow\mathbb{R}$ is convex, then
\begin{equation}
x^{\ast} \in\operatorname{Argmin}\{g\mid C\}\Longleftrightarrow0\in\partial
g(x^{\ast})+N_{C}(x^{\ast})\text{,}%
\end{equation}
(see \cite[Chapter VII, Theorem 1.1.1]{UL93}). In the case where $g$ is
continuously differentiable, it follows from the first order optimality
condition that%
\begin{equation}
x^{\ast} \in\operatorname{Argmin}\{g\mid C\}\Longleftrightarrow x^{\ast}\text{
solves VIP}(\nabla g,C),
\end{equation}
see, \textit{e.g.}, \cite[Proposition 3.1, p. 210]{BT89}.
\end{remark}

\begin{remark}%
\rm\
Both set-valued mappings, $\partial g$ and $N_{C}$, are maximal monotone and
$\operatorname{dom}\left(  \partial g\right)  =%
\mathbb{R}
^{n}$, hence also $\partial g+N_{C}$ is maximal monotone (see \cite[Corollary
24.4 (i)]{BC10}) and therefore
\begin{equation}
0\in(\partial g+N_{C})(x^{\ast})\Leftrightarrow x^{\ast}\in\operatorname{Fix}%
(J_{\lambda}^{\partial g+N_{C}})\text{.}%
\end{equation}

\end{remark}

Let's go back to the hierarchical optimization problem:%
\begin{equation}
\min\{f_{2}\mid\operatorname{Argmin}\{f_{1}\mid C\}\}\text{.} \label{eq:lex}%
\end{equation}
Under the assumption of the convexity of $f_{i}:\mathbb{R}^{n}\rightarrow
\mathbb{R}$, $i=1,2$, and the continuous differentiability of $f_{2}$, problem
(\ref{eq:lex}) can be reformulated as VIP$(\nabla f_{2},\operatorname{Fix}%
(J_{\lambda}^{\partial f_{1}+N_{C}}))$. That is, we look for a point $x^{\ast
}\in\operatorname{Fix}(J_{\lambda}^{\partial f_{1}+N_{C}})$ such that%
\begin{equation}
\left\langle \nabla f_{2}(x^{\ast}),x-x^{\ast}\right\rangle \geq0\text{ for
all }x\in\operatorname{Fix}(J_{\lambda}^{\partial f_{1}+N_{C}})\text{.}
\label{vip-fix}%
\end{equation}
So, under appropriate assumptions on $f:%
\mathbb{R}
^{n}\rightarrow%
\mathbb{R}
^{2}$, which assure that $\nabla f_{2}$ satisfies Conditions \ref{con:a} --
\ref{con:d}, we could apply Algorithm \ref{alg:ac}.

Next we present an example that can be translated into an appropriate VIP over
the fixed point set of a cutter operator.

\begin{example}%
\rm\
Let $C\subseteq%
\mathbb{R}
^{n}$ be closed and convex. Given a convex function $g:\mathbb{R}%
^{n}\rightarrow\mathbb{R}$, we are interested in minimizing $g$ over $C$ so
that the solution has minimal $p$-th norm, where $p\geq2$. This solution is
called a $p$-\texttt{minimal-norm solution}.

Define the operator $f=(f_{1},f_{2}):%
\mathbb{R}
^{n}\rightarrow%
\mathbb{R}
^{2}$ by%
\begin{equation}
f=\left(
\begin{array}
[c]{c}%
g\\
\frac{1}{p}\left\Vert \cdot\right\Vert _{p}^{p}+\frac{\alpha}{2}\left\Vert
\cdot\right\Vert _{2}^{2}%
\end{array}
\right)  \text{,}%
\end{equation}
where $\left\Vert \cdot\right\Vert _{p}$ denotes the $p$-th norm, i.e.,
$\left\Vert x\right\Vert _{p}:=\left(  \sum_{i=1}^{n}|x_{i}|^{p}\right)
^{1/p}$. Consider the following special case of problem (\ref{eq:lex}):
\begin{equation}
\left\{
\begin{array}
[c]{c}%
\text{minimize }\frac{1}{p}\Vert x\Vert_{p}^{p}+\frac{\alpha}{2}\Vert
x\Vert_{2}^{2}\\
\text{s.t. }x\in\operatorname{Argmin}\{g(x)\mid C\},
\end{array}
\right.
\end{equation}
which is a regularization of the problem under consideration. Notice that
$f_{2}$ is a sum of convex and strongly convex functions and, therefore,
$\nabla f_{2}$ is strongly monotone.

For $p=2$ we get that $\nabla(\frac{1}{2}\left\Vert \cdot\right\Vert _{2}%
^{2})=I$. Hence $\nabla f_{2}$ is Lipschitz continuous on $%
\mathbb{R}
^{n}$. Moreover, we do not need a regularization term to obtain strong
monotonicity and we can set $\alpha=0$. Therefore we can use Yamada's and
Ogura's hybrid steepest descent algorithm (see \cite[Section 4]{Y2004a}) to
solve VIP$(I,\operatorname{Fix}(J_{\lambda}^{\partial g+N_{C}}))$ and obtain a
2-minimal-norm solution.

Let now $p>2$. In this case $\nabla(\frac{1}{p}\left\Vert x\right\Vert
_{p}^{p})=(x_{1}|x_{1}|^{p-2},x_{2}|x_{2}|^{p-2},...,x_{n}|x_{n}|^{p-2})$. One
can easily check that $\nabla\frac{1}{p}\left\Vert \cdot\right\Vert _{p}^{p}$
is not globally Lipschitz continuous. Therefore we cannot use Yamada's and
Ogura's algorithm. However, we can use Algorithm \ref{alg:ac} to solve
VIP$(\nabla(\frac{1}{p}\left\Vert \cdot\right\Vert _{p}^{p}+\frac{\alpha}%
{2}\left\Vert \cdot\right\Vert _{2}^{2}),\operatorname{Fix}(J_{\alpha
}^{\partial g+N_{C}}))$. To see this, denote $\mathcal{F}_{1}:=\nabla(\frac
{1}{p}\left\Vert \cdot\right\Vert _{p}^{p})$, $\mathcal{F}_{2}:=\nabla
(\frac{\alpha}{2}\left\Vert \cdot\right\Vert ^{2})$ and $\mathcal{F}%
:=\mathcal{F}_{1}+\mathcal{F}_{2}$. By Remark \ref{remark:suffCond}, it
suffices to show that $\langle\mathcal{F}_{i}(x),x\rangle\geq c\Vert
\mathcal{F}_{i}(x)\Vert_{2}\Vert x\Vert_{2}$, for all $\Vert x\Vert_{2}\geq
R>0$, $i=1,2$, and some $c\in(0,1)$. Notice that for $i=2$ this inequality
holds for all $c\in(0,1)$, because
\begin{equation}
\langle\mathcal{F}_{2}(x),x\rangle=\alpha\Vert x\Vert_{2}^{2}=\Vert
\mathcal{F}_{2}(x)\Vert_{2}\cdot\Vert x\Vert_{2}.
\end{equation}
For $i=1$ the inequality is equivalent to
\begin{equation}
\left(  \Vert x\Vert_{p}^{p}\right)  ^{2}\geq c\cdot\Vert x\Vert_{2p-2}%
^{2p-2}\cdot\Vert x\Vert_{2}^{2},
\end{equation}
which follows directly from Lemma \ref{lema:normEquiv} (see the Appendix) with
$\alpha:=p-1$ and $\beta:=1$. By Remark \ref{remark:resolvent}, the operator
$T:=J_{\lambda}^{\partial g+N_{C}}$ is firmly nonexpansive and therefore it is
a cutter. Moreover, $I-T$ is closed at zero, by the nonexpansivity of $T$.
Hence Conditions \ref{con:a} -- \ref{con:d} are satisfied.
\end{example}

\section{Appendix}

\begin{lemma}
\label{lema:normEquiv} For any $\alpha,\beta\geq\frac{1}{2}$, there is
$c\in(0,1)$ such that
\begin{equation}
\left(  \sum_{i=1}^{n}|x_{i}|^{\alpha+\beta}\right)  ^{2}\geq c\sum_{i=1}%
^{n}|x_{i}|^{2\alpha}\cdot\sum_{i=1}^{n}|x_{i}|^{2\beta}%
\end{equation}
for all $x \in\mathbb{R}^{n}$.
\end{lemma}

\begin{proof}
We have to show the following inequality:
\begin{equation}
\left(  \Vert x\Vert_{\alpha+\beta}^{\alpha+\beta}\right)  ^{2}\geq c\Vert
x\Vert_{2\alpha}^{2\alpha}\cdot\Vert x\Vert_{2\beta}^{2\beta}.
\end{equation}
The norms $\left\Vert \cdot\right\Vert _{2\alpha}$, $\left\Vert \cdot
\right\Vert _{2\beta}$ and $\left\Vert \cdot\right\Vert _{\alpha+\beta}$ are
all equivalent; hence
\begin{equation}
\left\Vert x\right\Vert _{\alpha+\beta}\geq c\left\Vert x\right\Vert
_{2\alpha} \text{ and } \left\Vert x\right\Vert _{\alpha+\beta}\geq
c\left\Vert x\right\Vert _{2\beta}%
\end{equation}
for some $c>0$ and all $x\in\mathbb{R}^{n}$. Without any loss of generality,
we can assume that $c\in(0,1)$. Then
\begin{equation}
\left(  \Vert x\Vert_{\alpha+\beta}^{\alpha+\beta}\right)  ^{2}=\left\Vert
x\right\Vert _{\alpha+\beta}^{2\alpha}\cdot\left\Vert x\right\Vert
_{\alpha+\beta}^{2\beta}\geq c^{2\alpha+2\beta}\Vert x\Vert_{2\alpha}%
^{2\alpha}\cdot\Vert x\Vert_{2\beta}^{2\beta},
\end{equation}
which yields our assertion.
\end{proof}

\bigskip

\textbf{Acknowledgments}. The third author was partially supported by Israel
Science Foundation (ISF) Grant number 647/07, the Fund for the Promotion of
Research at the Technion and by the Technion VPR Fund.

\end{document}